\documentclass[11pt]{article}
\usepackage[margin=1in]{geometry}
\usepackage[utf8]{inputenc}
\usepackage[T1]{fontenc}
\usepackage{lmodern}
\usepackage{amssymb,amsmath,amsthm,amsfonts}
\usepackage{textgreek}
\usepackage{mathtools}
\usepackage{enumitem}
\usepackage[numbers,comma,sort&compress]{natbib}
\usepackage{authblk}
\usepackage{graphicx, caption, subcaption}
\usepackage{svg}
\usepackage{float}
\usepackage{algorithm}
\usepackage{algpseudocode}
\usepackage{qcircuit}
\usepackage{physics}
\usepackage{footnote}
\usepackage{xcolor}
\usepackage{mathrsfs}
\usepackage{bbm}
\usepackage{bm}
\usepackage{hhline}
\usepackage{cases}
\usepackage{url}
\usepackage[colorlinks,linkcolor=black,citecolor=cyan, urlcolor=cyan]{hyperref}
\usepackage{booktabs}
\usepackage{multirow}
\usepackage{graphicx}
\usepackage{colortbl} 
\usepackage{subcaption}   
\usepackage{tikz} 
\usepackage{comment}
\usepackage[normalem]{ulem} 

\theoremstyle{plain}
\newtheorem{theorem}{Theorem}
\newtheorem{proposition}[theorem]{Proposition}
\newtheorem{lemma}[theorem]{Lemma}
\newtheorem{corollary}[theorem]{Corollary}
\theoremstyle{definition}
\newtheorem{definition}{Definition}

\newtheorem{remark}[theorem]{Remark}

\usepackage{thm-restate}

\newcommand{\eqn}[1]{\hyperref[eqn:#1]{(\ref*{eqn:#1})}}
\newcommand{\rem}[1]{\hyperref[rem:#1]{Remark~\ref*{rem:#1}}}
\newcommand{\thm}[1]{\hyperref[thm:#1]{Theorem~\ref*{thm:#1}}}
\newcommand{\cor}[1]{\hyperref[cor:#1]{Corollary~\ref*{cor:#1}}}
\newcommand{\defn}[1]{\hyperref[defn:#1]{Definition~\ref*{defn:#1}}}
\newcommand{\lem}[1]{\hyperref[lem:#1]{Lemma~\ref*{lem:#1}}}
\newcommand{\prop}[1]{\hyperref[prop:#1]{Proposition~\ref*{prop:#1}}}
\newcommand{\fig}[1]{\hyperref[fig:#1]{Figure~\ref*{fig:#1}}}
\newcommand{\tab}[1]{\hyperref[tab:#1]{Table~\ref*{tab:#1}}}
\newcommand{\algo}[1]{\hyperref[algo:#1]{Algorithm~\ref*{algo:#1}}}
\renewcommand{\sec}[1]{\hyperref[sec:#1]{Section~\ref*{sec:#1}}}
\newcommand{\secapp}[1]{\hyperref[sec:#1]{Appendix~\ref*{sec:#1}}}
\newcommand{\append}[1]{\hyperref[append:#1]{Appendix~\ref*{append:#1}}}
\newcommand{\fac}[1]{\hyperref[fac:#1]{Fact~\ref*{fac:#1}}}
\newcommand{\lin}[1]{\hyperref[lin:#1]{Line~\ref*{lin:#1}}}
\newcommand{\fnote}[1]{\hyperref[fnote:#1]{Footnote~\ref*{fnote:#1}}}

\newcommand{\prob}[1]{\hyperref[prob:#1]{Problem~\ref*{prob:#1}}}
\newcommand{\assump}[1]{\hyperref[assump:#1]{Assumption~\ref*{assump:#1}}}

\def\>{\rangle}
\def\<{\langle}
\def\trans{^{\top}}

\newcommand{\sdg}{Schr\"{o}dinger}

\newcommand{\Z}{\mathbb{Z}}
\newcommand{\R}{\mathbb{R}}
\newcommand{\C}{\mathbb{C}}

\renewcommand{\d}{\mathrm{d}}

\DeclareMathOperator{\poly}{poly}
\DeclareMathOperator{\polylog}{polylog}

\DeclareMathOperator*{\argmin}{arg\,min}

\numberwithin{equation}{section} 

\newcommand{\xx}{\boldsymbol{x}}
\newcommand{\yy}{\boldsymbol{y}}
\newcommand{\vv}{\boldsymbol{v}}
\renewcommand{\gg}{\boldsymbol{g}}
\newcommand{\kk}{\boldsymbol{k}}

\DeclareMathOperator{\conv}{\mathrm{conv}}

\begin{document}

\title{Quantum Hamiltonian Descent for Non-smooth Optimization}

\author[1,2,$\dagger$]{Jiaqi Leng\thanks{\href{mailto:jiaqil@berkeley.edu}{jiaqil@berkeley.edu}}}
\author[3,4,$\dagger$]{Yufan Zheng}
\author[5]{Zhiyuan Jia}
\author[6,7]{Lei Fan}
\author[8]{Chaoyue Zhao}
\author[9]{Yuxiang Peng}
\author[3,4]{Xiaodi Wu}

\affil[1]{Simons Institute for the Theory of Computing, University of California, Berkeley}
\affil[2]{Department of Mathematics, University of California, Berkeley}
\affil[3]{Department of Computer Science, University of Maryland, College Park}
\affil[4]{Joint Center for Quantum Information and Computer Science, University of Maryland}
\affil[5]{Department of Applied Mathematics, University of Washington}
\affil[6]{Department of Engineering Technology, University of Houston}
\affil[7]{Department of Electrical and Computer Engineering, University of Houston}
\affil[8]{Department of Industrial and Systems Engineering, University of Washington}

\affil[9]{Artephi Computing Inc.}
\affil[$\dagger$]{Equal Contribution}
\date{}
\maketitle

\begin{abstract}
Non-smooth optimization models play a fundamental role in various disciplines, including engineering, science, management, and finance. However, classical algorithms for solving such models often struggle with convergence speed, scalability, and parameter tuning, particularly in high-dimensional and non-convex settings. In this paper, we explore how quantum mechanics can be leveraged to overcome these limitations. Specifically, we investigate the theoretical properties of the Quantum Hamiltonian Descent (QHD) algorithm for non-smooth optimization in both continuous and discrete time. 
First, we propose continuous-time variants of the general QHD algorithm and establish their global convergence and convergence rate for non-smooth convex and strongly convex problems through a novel Lyapunov function design. Furthermore, we prove the finite-time global convergence of continuous-time QHD for non-smooth non-convex problems under mild conditions (i.e., locally Lipschitz). 
In addition, we propose discrete-time QHD, a fully digitized implementation of QHD via operator splitting (i.e., product formula). We find that discrete-time QHD exhibits similar convergence properties even with large time steps.
Finally, numerical experiments validate our theoretical findings and demonstrate the computational advantages of QHD over classical non-smooth non-convex optimization algorithms.
\end{abstract}

\newpage 
\tableofcontents
\newpage

\section{Introduction}
Non-smooth optimization involves the study of functions that lack differentiability at certain points or across entire regions of their domains. This lack of smoothness often arises in practical problems where robustness, sparsity, or discrete decision-making is required. Non-smooth optimization plays a vital role in various fields such as machine learning, signal processing, and finance. For example, in classification problems, techniques like support vector machines use hinge loss functions that are inherently non-smooth \cite{astorino2008non}. Similarly, robust regression techniques leverage $L_1$ norm regularization to enforce sparsity and manage outliers, which are particularly relevant in fields like finance and bioinformatics \cite{shamir2013stochastic}. Engineering and physics frequently involve variational problems where non-smooth constraints model real-world constraints effectively \cite{kovtunenko2022non}. In addition, imaging problems such as deblurring and denoising often rely on non-smooth objectives to enhance computational efficiency and accuracy \cite{chambolle2016introduction}.

We will consider the following unconstrained minimization problem:
\begin{align}\label{eqn:problem-form}
    \min_{\xx \in \R^d} f(\xx),
\end{align}
where the real-valued function $f\colon \R^d \to \R$ is assumed to be continuous but not necessarily differentiable (i.e., non-smooth). For ease of theoretical treatment, we also assume that the function $f$ is finite over $\R^d$ and has a unique global minimum $f^*$ attained at the point $\xx^*$.

Classical algorithms, such as subgradient and proximal gradient methods, have been widely adopted to address this problem. Subgradient methods \cite{boyd2003subgradient, nedic2001incremental, duchi2011adaptive} generalize gradients for non-differentiable functions and are applicable across diverse optimization scenarios. However, their convergence rates can be slow, particularly in high-dimensional or non-convex problems \cite{shamir2013stochastic}. Despite being foundational, their scalability remains limited in practice, as confirmed by broader research on efficient subgradient applications \cite{bullins2019higher}.
Proximal gradient methods \cite{li2015accelerated,sahu2021convergence} are particularly suited for composite optimization problems, where they combine gradient descent with proximal operators to handle non-smooth terms. These methods have proven effective in specific contexts, such as sparse modeling and total variation minimization \cite{stella2017proximal}. However, similar to subgradient methods, they often struggle with parameter tuning and efficiency in non-convex scenarios \cite{nesterov2005smooth}.

Beyond these, techniques like smoothing and bundle methods attempt to bridge gaps in classical approaches. Smoothing methods enhance computational tractability and convergence, particularly for problems with inherent non-smoothness \cite{kreimer1992nondifferentiable}. Bundle methods focus on large-scale optimization challenges by leveraging memory-efficient strategies \cite{haarala2004new}. Despite these advancements, inefficiencies in handling non-convexity and high-dimensionality remain prominent issues, underscoring the importance of exploring newer paradigms. 

In recent years, quantum computing has emerged as a transformative technology that reshapes the landscape of computational methods. Quantum computers have demonstrated exponential speedups for fundamental problems in scientific computing, such as solving eigenvalue problems~\cite{kitaev1995quantum} and linear systems~\cite{harrow2009quantum}, as well as simulating differential equations~\cite{berry2014high,berry2017quantum}.
Building on quantum subroutines for numerical computation, numerous quantum algorithms for optimization have been developed and shown to be particularly promising in high-dimensional regimes, where classical approaches often face challenges related to robustness and efficiency (e.g., ~\cite{rebentrost2019quantum,kerenidis2020quantum,chakrabarti2020quantum,van2020convex, augustino2023quantum,mohammadisiahroudi2022efficient}).

However, despite the extensive literature on quantum algorithms for differentiable and convex optimization problems, quantum algorithms for non-smooth problems remain relatively underexplored.~\citet{garg2021no} show that there is no quantum speedup over classical methods for non-smooth convex optimization in the worst-case scenario. In contrast, a recent work by~\citet{liu2024quantum} presents a quantum algorithm for finding a Goldstein stationary point of Lipschitz continuous functions that outperforms the best-known classical algorithms by a polynomial factor, pioneering the investigation of quantum advantage in non-smooth, non-convex optimization. Overall, the development of efficient quantum algorithms for non-smooth optimization—ones that achieve both accelerated convergence and improved solution quality—remains an open avenue for pursuing quantum advantage.

Quantum Hamiltonian Descent (QHD)~\cite{leng2023qhd} is a recently proposed quantum algorithm for continuous optimization. 
In this approach, the solution to an optimization problem is encoded as low-energy configurations of certain quantum-mechanical systems, and the search for the function minimum is implemented as an evolution of a time-dependent quantum system.
This method is inspired by a (classical) Hamiltonian formulation of the celebrated Nesterov's accelerated gradient descent algorithm~\cite{su2016differential,wibisono2016variational}, and the convergence of QHD is proven for smooth convex objective functions.
Follow-up work shows that QHD can find the global solution to a class of high-dimensional unconstrained optimization problems, each containing exponentially many local minima, in a small polynomial time~\cite{leng2023qhd}. This result underscores a significant performance advantage of QHD for solving smooth non-convex problems over widely used classical methods, such as first- and second-order approaches, which do not appear to solve these problems within polynomial time.

An interesting feature of QHD is that it solely relies on function values without requiring gradient information.
Intuitively, this feature makes QHD particularly suitable for non-smooth optimization, where the gradients may not exist. However, the convergence analysis of QHD~\cite{leng2023qhd,leng2023quantum} heavily relies on the differentiability of the objective function and cannot be trivially generalized to non-smooth optimization.
Moreover, the original QHD and its convergence analysis involve abstract time-dependent parameters, making it unclear how to choose these parameters to effectively address various non-smooth optimization problems in practice.
It remains an open question whether QHD is effective for non-smooth optimization and, if so, what potential quantum advantages it may offer over classical approaches such as subgradient methods.

\paragraph{Our results.}
In this paper, we study the performance of QHD for non-smooth optimization problems in both continuous and discrete time settings.
Our main contributions are summarized as follows:
\begin{itemize}
    \item We propose three variants of QHD, each formulated as a differential equation to address specific classes of non-smooth optimization problems with distinct convexity structures: \textit{(i)} \textbf{QHD-SC} for non-smooth strongly convex problems, \textit{(ii)} \textbf{QHD-C} for non-smooth convex problems, and \textit{(iii)} \textbf{QHD-NC} for non-smooth non-convex problems. These variants, detailed in~\sec{continuous-time-qhd-formulation}, are collectively referred to as \textbf{continuous-time QHD}. The corresponding continuous-time dynamics can be directly employed as quantum optimization algorithms by simulating their time evolution on a quantum computer (see~\algo{cont-time-qhd}). We analyze the query and gate complexity of these algorithms in~\cor{continuous-time-qhd-complexity}.
    \item We establish the global convergence of continuous-time QHD using energy reduction arguments tailored to each problem class. The convergence rates are proven for strongly convex (\thm{qhd-strongly-convex}) and generally convex (\thm{qhd-convex}) optimization problems. We also prove that \textbf{QHD-NC} converges to the global minimum of any (possibly non-smooth and non-convex) locally Lipschitz objective function in finite time (\thm{qhd-non-convex}), which, to the best of our knowledge, is the \textit{first rigorous} result in the literature on quantum optimization achieving such a guarantee. In contrast, classical methods like the subgradient method can fail to converge under the locally Lipschitz condition~\cite{daniilidis2020pathological,rios2022examples}, and no convergence result is known for the stochastic subgradient method assuming the locally Lipschitz condition only\textemdash let alone convergence to the global minimum.\footnote{Exceptions exist when there are more assumptions on objective functions (e.g., being semialgebraic or definable in an o-minimal structure), under which the stochastic subgradient method converges to first-order stationary points~\cite{davis2020stochastic}.}
    \item Inspired by the operator splitting (or product formula) technique in quantum simulation, we develop three \textbf{discrete-time QHD} algorithms (see~\algo{discrete-time-qhd}), each corresponding to a variant of continuous-time QHD. The query and gate complexity of these discrete-time algorithms are analyzed in~\prop{discrete-time-qhd-complexity}. Similar to classical gradient descent, these algorithms iteratively approach the solution using constant step sizes (i.e., the length of time steps). Interestingly, numerical studies show that they continue to exhibit convergence behavior even when the step size is so large that the discrete-time iterations no longer follow the continuous-time dynamics.
    \item Additionally, we perform numerical experiments to evaluate the performance of QHD in non-smooth non-convex optimization. We observe that QHD outperforms the subgradient method and its stochastic variants in 10 out of 12 test problems. In some cases, QHD achieves an optimality gap that is \textit{orders of magnitude} lower than those of the classical methods. Our findings make QHD a promising approach to addressing non-smooth optimization instances in both theoretical and practical applications.
\end{itemize}

\begin{table}[htbp]
    \centering
    \resizebox{\columnwidth}{!}{
        \begin{tabular}{lccc}
            \toprule
            & Continuous-Time QHD & Discrete-Time QHD & (Stochastic) Subgradient Method \\
            \midrule
            $\mu$-Strongly Convex               & $O(e^{-\sqrt{\mu}t})$ (\thm{qhd-strongly-convex}) & $O(t^{-1})$ (\sec{comparison-subgrad}) & $\Theta(t^{-1} \log t)$~\cite{bubeck2015convex,hazan2007logarithmic,rakhlin2012making} \\
            Convex                        & $O(t^{-2})$ (\thm{qhd-convex}) & $O(t^{-2/3})$ (\sec{comparison-subgrad}) & $\Theta(t^{-1/2})$~\cite{beck2017first,nemirovski2009robust,nemirovski1983problem} \\
            Locally Lipschitz             & Global Minimum (\thm{qhd-non-convex}) & Global Minimum (\thm{qhd-non-convex})\footnotemark & Convergence Not Guaranteed~\cite{daniilidis2020pathological,rios2022examples} \\
            \bottomrule
        \end{tabular}
    }
    \caption{Comparison of (worst-case) convergence rate bounds and convergence guarantees for continuous-time QHD, discrete-time QHD, and the subgradient method across various non-smooth function classes. Here, \(t\) denotes time or the iteration number. Discrete-time QHD bounds are obtained via numerical experiments.}
    \label{tab:comparison_methods}
\end{table}

\footnotetext{Given sufficiently small step size, discrete-time QHD follows the same convergence trajectory as continuous-time QHD.} 

In~\tab{comparison_methods}, we present the convergence results for both continuous-time and discrete-time QHD across various classes of non-smooth functions. For comparison, we also include the convergence rates of the subgradient method in both deterministic and stochastic settings. Continuous-time QHD achieves a faster convergence rate (or guarantees global convergence) compared to the subgradient method in all three categories. While discrete-time QHD does not provide a speedup over classical methods for strongly convex functions, it demonstrates advantages for generally convex and non-convex optimization problems.

The discrepancy in convergence rates between continuous- and discrete-time QHD arises from the fact that discrete-time QHD algorithms eventually reach a nonzero sub-optimality gap that depends on the step size. Before this gap saturates, the discrete-time algorithms exhibit the same convergence rate as their continuous-time counterparts. This newly observed behavior, which does not occur in smooth optimization problems, allows us to postulate the iteration complexity of discrete-time QHD. A rigorous justification of this complexity is beyond the scope of this work and is left for future research.

\paragraph{Technical summary.}
In this work, we analyze the computational complexity of both continuous- and discrete-time QHD, assuming a standard gate-model quantum computer. The complexity analysis of continuous-time QHD relies on a strengthened version (\prop{spectral-method}) of the time-dependent Hamiltonian simulation algorithm in~\cite{childs2022quantum}. For discrete-time QHD, the quantum implementation is enabled by efficiently diagonalizing the kinetic operator via the Quantum Fourier Transform, as detailed in~\prop{discrete-time-qhd-complexity}.
We also provide a convergence analysis of continuous-time QHD for various classes of non-smooth optimization problems. A brief summary follows below.
\begin{itemize}
    \item Non-smooth convex optimization: In the seminal work~\cite{leng2023qhd}, the convergence results are obtained by constructing a specific Lyapunov function. Mathematically, a Lyapunov function is a time-dependent function associated with a dynamical system that is non-increasing in time. Lyapunov analysis is a standard tool in the study of long-term behaviors of differential equations. By showing the Lyapunov function is non-increasing in time, \cite{leng2023qhd} proves the expectation value of the objective converges to the global minimum $f(x^*)$ at a rate $\mathcal{O}(e^{-\beta_t})$, where $\beta_t$ is a time-dependent function in the quantum Hamiltonian. A detailed overview of QHD is provided in~\sec{qhd-review}. 
    However, the choice of $\beta_t$ (and other time-dependent parameters) is not clear in our setting. Meanwhile, to show the monotonicity of the Lyapunov function, \cite{leng2023qhd} requires a well-defined derivative of the objective $f$, which does not directly apply to a non-smooth objective $f$. 
    To circumvent the technical issue, we
    \textit{(i)} introduce two Lyapunov functions, each tailored to \textbf{QHD-C} and \textbf{QHD-SC}, respectively, and then \textit{(ii)} employ a weaker notion of derivatives to rewrite the operator commutator in the Lyapunov function via the integral by parts formula (see~\lem{weak-der-c}).
    It is worth noting that our new argument requires a $C^\infty$ regularity estimate of the solution to the Schr\"odinger equation, as discussed in~\append{regularity}.
    
    \item Non-smooth non-convex optimization:~\cite{leng2023qhd} provides a heuristic argument to show the global convergence of QHD based on a quantum adiabatic approximation for smooth non-convex optimization problems. This heuristic argument heavily relies on the fact that the function $f$ is twice differentiable at the global minimizer $x^*$, which may not hold in non-smooth optimization. 
    We rigorously establish the asymptotic convergence of \textbf{QHD-NC} to the global minimum of $f(x)$ by exploiting a correspondence (often called an \emph{energy argument} in physics) between the expected function value $\mathbb{E}[f(X_t)]$ and the energy spectrum of the quantum Hamiltonian $H_{\mathrm{NC}}(t)$. We prove the global convergence of \textbf{QHD-NC} by analyzing the spectral properties of the quantum Hamiltonian via Weyl's law~\cite{ivrii2016100,rozenblum1976distribution}, which enables us to get rid of the smoothness assumption at the global minimizer.
\end{itemize}

\paragraph{Organization.}
This paper is organized as follows. \sec{related-work} reviews existing quantum optimization algorithms. \sec{qhd-review} briefly reviews the original QHD algorithm, providing the necessary theoretical background for the rest of the work. 
\sec{convergence-analysis} presents the three variants of continuous-time QHD for various classes of optimization problems and provides the corresponding convergence analysis.
\sec{discrete-time-qhd} introduces discrete-time QHD and investigates its convergence behaviors for non-smooth optimization problems.
\sec{num} presents a numerical evaluation of QHD for non-smooth non-convex optimization problems and demonstrates its advantage over classical algorithms.

\paragraph{Acknowledgment.} JL is partially supported by the Simons Quantum Postdoctoral Fellowship, DOE QSA 
grant \#FP00010905, and a Simons Investigator award through Grant No. 825053. XW and YZ are partially supported by NSF CAREER Award CCF-1942837, a Sloan research fellowship, and the U.S. Department of Energy, Office of Science, Accelerated Research in Quantum Computing, Fundamental Algorithmic Research toward Quantum Utility (FAR-Qu).

\section{Related work}\label{sec:related-work}
\paragraph{Quantum algorithms for discrete optimization.}
The first applications of quantum algorithms have predominantly focused on discrete optimization, as most of these algorithms can be implemented on currently available quantum architectures. Notable examples include Grover's Search~\cite{grover1996fast,durr1996quantum,bulger2003implementing,baritompa2005grover},  Variational Quantum Eigensolver~\cite{tilly2022variational,peruzzo2014variational,mcclean2016theory,cerezo2021variational}, Quantum Adiabatic Algorithm~\cite{mcgeoch2014adiabatic,kadowaki1998quantum,farhi2000quantum,ronnow2014defining} and Quantum Approximate Optimization Algorithm~\cite{blekos2024review,zhou2020quantum,farhi2014quantum}. More recent breakthroughs, such as the Short-Path Algorithm~\cite{hastings2018short}, which outperforms Grover's Search with a theoretical guarantee, and Decoded Quantum Interferometry~\cite{jordan2024optimization}, which reduces optimization problems to \emph{classical} decoding tasks, have opened new avenues for exploring quantum optimization.

\paragraph{Quantum algorithms for continuous optimization.}
In recent decades, significant progress has been made in developing quantum algorithms for continuous optimization. 
Quantum Gradient Descent methods implement gradient descent in $n$-dimensional space using only $\polylog(n)$ qubits and they are efficient if the number of iterations is small~\cite{rebentrost2019quantum} or the gradient is an affine function~\cite{kerenidis2020quantum}.
A quadratic speedup (in the instance dimension $n$) has been achieved for general convex optimization~\cite{chakrabarti2020quantum,van2020convex}. For more structured problems, such as linear programming and semidefinite programming, faster quantum algorithms have been devised using Multiplicative Weights Update (MWU) framework~\cite{brandao2017quantum,brandao2017quantum2,van2019improvements,van2019quantum} or interior-point methods~\cite{kerenidis2020quantum2,mohammadisiahroudi2022efficient,augustino2023quantum,wu2023inexact,apers2023quantum}.  
In contrast, much less is known about \emph{provable} quantum speedups for non-convex optimization. A mild quantum speedup has been demonstrated for finding critical points in stochastic optimization~\cite{sidford2024quantum}, while a cubic speedup (in the instance dimension $n$) has been achieved for escaping saddle points~\cite{zhang2021quantum}. 

Most of the algorithms discussed above share a common feature: they leverage quantum acceleration in intermediate steps of classical optimization methods.  
Recently, a new research direction has emerged, focusing on directly harnessing quantum dynamics in carefully designed systems for optimization. Notable examples include QHD  and Quantum Langevin Dynamics (QLD)~\cite{chen2023quantum}. These methods provably converge to global minima on smooth objectives, albeit not always efficiently, and have demonstrated strong empirical performance in tackling non-convex optimization problems.

For a comprehensive overview of quantum optimization, we refer readers to the survey by Abbas et al.~\cite{abbas2024challenges}.

\section{An Overview of Quantum Hamiltonian Descent}\label{sec:qhd-review}

Quantum Hamiltonian Descent (QHD) is a quantum optimization algorithm motivated by the interplay between classical gradient-based optimization algorithms and differential equations. 
The formulation of QHD is closely related to the continuous-time limit of Nesterov's accelerated gradient descent (NAG) method. 
In this section, we provide an overview of the QHD algorithm, including its origin from classical accelerated gradient descent methods, continuous-time formulation, and discrete-time implementation. 

\subsection{Hamiltonian formulation of accelerated gradient descent}

\paragraph{NAG and differential equations.}
There is a well-established historical connection between ordinary differential equations (ODEs) and optimization. For example, the standard gradient descent (GD) can be regarded as the forward Euler discretization of the gradient flow equation:
\begin{align}
    \dot{X}_t = - \nabla f(X_t), \quad X_t = x_0.
\end{align}
where $x_0$ is the initial guess. 
Accelerated gradient descent algorithms are first-order methods that achieve a convergence rate faster than the standard GD.
Nesterov~\cite{nesterov1983method} proposed the first accelerated gradient descent algorithm (known as Nesterov's accelerated gradient descent, or NAG).
The update rules of NAG are given as follows:
\begin{subequations}
    \begin{align}
        x_k &= y_{k-1} - s\nabla f(y_{k-1}),\\
        y_k &= x_k + \frac{k-1}{k+2} (x_k - x_{k-1}),
    \end{align}
\end{subequations}
where $x_0$ is an initial guess, $y_0 = 0$, and $s > 0$ is a fixed step size. 
For decades, the mechanism of acceleration in NAG appeared mysterious, and an extensive body of research works has been dedicated to a more intuitive understanding of accelerated gradient descent algorithms.
In a seminal work~\cite{su2016differential}, the authors show that the continuous-time limit of NAG can be described as a simple second-order ordinary differential equation (ODE):
\begin{align}\label{eqn:nesterov-ode}
    \ddot{X}_t + \frac{3}{t}\dot{X}_t + \nabla f(X_t) = 0, \quad X_t = x_0.
\end{align}
This continuous-time dynamical system perspective of NAG has revolutionized the understanding and design of accelerated gradient descent through systematically introducing of techniques from continuous mathematics, in particular, the theory of differential equations. 

The ODE perspective of NAG was later extended to a broader framework that characterizes accelerated gradient descent by Hamiltonian dynamics~\cite{wibisono2016variational}. Specifically, the Hamiltonian function that models NAG is the following,
\begin{align}
    H(X, P, t) = \frac{1}{2t^3}\|P\|^2 + t^3 f(X),
\end{align}
where $X$ and $P$ represent the position and momentum variables of the system, respectively.
The Hamiltonian flow generated by the function $H(X_t,P_t,t)$ is described by the following system of ODEs:
\begin{subequations}\label{eqn:nesterov-ham}
    \begin{align}
     \dot{X}_t &= \frac{\partial H}{\partial P_t} = \frac{1}{t^3}P_t,\label{eqn:nesterov-ham-a}\\
    \dot{P}_t &= - \frac{\partial H}{\partial X_t} = - t^3 \nabla f(X_t),\label{eqn:nesterov-ham-b}
    \end{align}
\end{subequations}
with initial data $X_0 = x_0$ and $P_0 = 0$. It can be readily verified that the Hamiltonian flow is equivalent to the ODE model~\eqn{nesterov-ode} by plugging~\eqn{nesterov-ham-b} into~\eqn{nesterov-ham-a}.

\paragraph{The Bregman-Lagrangian framework.}
Motivated by the connection between NAG and Hamiltonian flows, \citet{wibisono2016variational} proposed a more general framework to model accelerated gradient descent using Hamiltonian/Lagrangian mechanics. This framework is based on the so-called Bregman Lagrangian:
\begin{align}
    \mathcal{L}(X,\dot{X},t) = e^{\alpha_t + \gamma_t}\left(\frac{1}{2}|e^{-\alpha_t }\dot{X}|^2  - e^{\beta_t}f(X)\right),
\end{align}
where $f(x)\colon \R^d \to \R$ is the objective function, $t \ge 0$ represents the time variable, $X \in \R^d$ is the position, $\dot{X} \in \R^d$ is the velocity, and $\alpha_t, \beta_t, \gamma_t$ are arbitrary smooth functions that control the damping of energy in the system. The Euler-Lagrange equation associated with the Bregman Lagrangian reads:
\begin{align}\label{eqn:second-order-ode}
    \ddot{X}_t + (\dot{\gamma}_t - \dot{\alpha}_t) \dot{X}_t + e^{2\alpha_t+\beta_t} \nabla f(X_t)=0,
\end{align}
which is a second-order differential equation. In particular, the ODE model of NAG~\eqn{nesterov-ode} is a special case of the Euler-Lagrange equation~\eqn{second-order-ode} by choosing $\alpha_t = -\log(t)$ and $\beta_t = \gamma_t = 2\log(t)$.
Equivalently, we can write down the Hamiltonian function corresponding to the Bregman Lagrangian via the Legendre transformation~\cite[Section 3]{wibisono2016variational},
\begin{align}\label{eqn:bregman-hamiltonian}
    \mathcal{H}(X,P,t) = e^{\alpha_t+\gamma_t}\left(\frac{1}{2}|e^{-\gamma_t }P|^2 + e^{\beta_t}f(X)\right),
\end{align}
where $P$ and $X$ represent the momentum and position variables, respectively.
The Euler-Lagrange equation~\eqn{second-order-ode} can be directly recovered from the Hamilton's equations associated with the Hamiltonian~\eqn{bregman-hamiltonian}.

The convergence of the Euler-Lagrange equation can be established via an energy reduction argument. First, we consider the following energy function 
\begin{align}
    \mathcal{E}(t) = \|e^{-\alpha_t} \dot{X}_t + X_t - x^*\|^2 + e^{\beta_t}\left(f(X_t) - f(x^*)\right).
\end{align}
Under the so-called \textit{ideal scaling conditions}, i.e., 
\begin{align}\label{eqn:ideal-scaling}
    \dot{\beta}_t \le e^{\alpha_t},\quad \dot{\gamma}_t = e^{\alpha_t},
\end{align}
we can prove that the energy function $\mathcal{E}(t)$ is non-increasing in time:
\begin{align}
    \mathcal{E}(t) \le \mathcal{E}(0)\quad \forall t \ge 0.
\end{align}
In other words, $\mathcal{E}(t)$ is a \textit{Lyapunov function} associated with the dynamical system. As a result, we obtain the convergence rate of the Euler-Lagrange equation~\cite[Theorem 1.1]{wibisono2016variational}:
\begin{align}
    f(X_t) - f(x^*) \le \mathcal{O}(e^{-\beta_t}).
\end{align}

\subsection{Quantum dynamics for optimization}
\paragraph{Quantization of accelerated gradient descent.}
Hamiltonian evolution serves as a bridge between two fundamental realms of physics: classical and quantum mechanics.
Through this connection, the classical Hamiltonian formulation of accelerated gradient descent has inspired \textbf{Quantum Hamiltonian Descent (QHD)}~\cite{leng2023qhd}, which is expressed as a quantum Hamiltonian evolution governed by the following partial differential equation (PDE),
\begin{align}\label{eqn:qhd-review}
    i \partial_t \Psi(t,x) = \hat{H}(t) \Psi(t,x),\quad \hat{H}(t) = e^{\alpha_t - \gamma_t}\left(-\frac{1}{2}\Delta\right) + e^{\alpha_t+\beta_t+\gamma_t} f(x),
\end{align}
Here, $i = \sqrt{-1}$ is the imaginary unit, and $\Delta = \sum^d_{j=1}\partial^2_{x_j}$ is the Laplacian operator.
The operator $\hat{H}(t)$ is called the quantum Hamiltonian, which represents the total energy of the system. The quantum Hamiltonian $\hat{H}(t)$ in QHD can be regarded as the direct quantization of the Bregman Hamiltonian~\eqn{bregman-hamiltonian}, where we invoke the following change of variables (a procedure known as \textit{canonical quantization}~\cite{takhtadzhian2008quantum}):
\begin{align}
    P_j \mapsto \hat{p}_j = -i \partial_{x_j}, \quad X_j \mapsto \hat{x}_j, \quad \forall j \in [d].
\end{align}
Note that $\hat{p}_j$ and $\hat{x}_j$ should be interpreted as operators rather than  scaler functions. In particular, given a smooth test function $\varphi \colon \R^d \to \R$, we have
\begin{align}
    (\hat{p}_j \varphi)(x) \coloneqq -i \partial_{x_j} \varphi(x),\quad (\hat{x}_j \varphi)(x) \coloneqq x_j \varphi(x).
\end{align}

The function $\Psi(t,\xx) \colon [0, \infty)\times \R^d \to \C$ is the quantum wave function that describes the state of the quantum system at time $t\ge 0$. 
It is known that the evolutionary PDE~\eqn{qhd-review} preserves the $L^2$-norm of the wave function. In other words, given an initial state $\Psi(0,\xx) = \Psi_0(\xx)$ with a unit $L^2$-norm, i.e., $\|\Psi_0\| = 1$, the wave function always has a unit $L^2$-norm during the time evolution:
\begin{align}
    \int_{\R^d} |\Psi(t,\xx)|^2~\d \xx = \int_{\R^d} |\Psi_0(\xx)|^2~\d \xx =  1, \quad \forall t \ge 0.
\end{align}
Therefore, the modulus square of the wave function $|\Psi(t,\xx)|^2$ can be interpreted as the probability density of the quantum particle in the real space $\R^d$ at a certain time $t$.

Similar to the Bregman-Lagrangian framework, the asymptotic convergence of the quantum PDE can be shown under the ideal scaling conditions~\eqn{ideal-scaling}. We consider the following quantum energy function:
\begin{align}
    \mathcal{W}(t) \coloneqq \int_{\R^d} \overline{\Psi(t,\xx)} \hat{O}(t) \Psi(t,\xx)~\d \xx,\quad \hat{O}(t) \coloneqq \frac{1}{2}\sum^d_{j=1}\|e^{-\gamma_t}\hat{p}_j + \hat{x}_j\|^2 + e^{\beta_t} f(x).
\end{align}
For a smooth convex objective function $f$ such that $f(\xx^*) = 0$ and $\xx^* = 0$, we can prove that $\mathcal{W}(t)$ is non-increasing in time. As a consequence, we obtain the convergence rate of QHD~\cite[Theorem 1]{leng2023qhd}:
\begin{equation}
    \mathbb{E}[f(X_t)] - f(\xx^*) \le \mathcal{O}(e^{-\beta_t}),
\end{equation}
where $X_t$ is a random variable distributed following the probability density function $|\Psi(t,\xx)|^2$.

\paragraph{QHD as quantum optimization algorithms.}
The quantum dynamics~\eqn{qhd-review} can be exploited to solve optimization problems by simulating the time evolution using a quantum computer.
Simulating quantum Hamiltonian evolution is a fundamental task in quantum computing. It can be shown that, under reasonable assumptions, the QHD dynamics can be efficiently simulated on gate-based quantum computers with $\tilde{\mathcal{O}}(T)$ quantum queries to the function $f$ and $\tilde{\mathcal{O}}(dT)$ total gate complexity~\cite{childs2022quantum,leng2023quantum}, where $d$ is the problem dimension and $T$ is the total evolution time.
Alternatively, it is also possible to simulate the QHD dynamics using an analog quantum Ising Hamiltonian simulator via the \textit{Hamiltonian embedding} technique~\cite{leng2024expanding}. This approach leads to resource-efficient implementations of QHD to solve large-scale non-linear and non-convex optimization problems on near-term quantum devices~\cite{leng2023qhd,kushnir2025qhdopt}.

In this paper, we focus on the implementation of QHD using gate-based quantum computers. In~\algo{cont-time-qhd}, we formalize the quantum algorithm based on simulating the continuous-time QHD dynamics. A rigorous complexity analysis is provided in~\cor{continuous-time-qhd-complexity}.

\section{Continuous-time QHD for Non-Smooth Optimization}\label{sec:convergence-analysis}

\subsection{Quantum algorithms and complexity analysis}\label{sec:continuous-time-qhd-formulation}

Motivated by the classical study of accelerated gradient descent~\cite{shi2022understanding}, we propose three variants of QHD, each corresponding to a specific class of optimization problems. In each class, the optimization problems are characterized by the convexity of the objective function.

\begin{definition}[Convexity]
    A function $f\colon \R^d \to \R$ is \textit{convex} if for all $0 \le t \le 1$ and any $\xx, \yy \in \R^d$, we have
\begin{equation}
    f(t\xx + (1-t)\yy) \le tf(\xx) + (1-t)f(\yy).
\end{equation}
\end{definition}

\begin{definition}[Subgradient]
    Let $f\colon \R^d \to \R$. We say a vector $\gg\in \R^d$ is a \textit{subgradient} of $f$ at $\xx$ if,
\begin{equation}
    f(\yy) \ge f(\xx) + \gg\trans (\yy-\xx), \quad \forall \yy \in \R^d.
\end{equation}
\end{definition}

The subgradient of a function $f$ at a fixed point is not unique. The \textit{subdifferential} of $f$ at $\xx \in \R^d$ is defined as the set of all subgradients of $f$ at $\xx$, denoted by
\begin{equation}
    \partial f(\xx) \coloneqq \{\gg\in\R^d\colon f(\yy) \ge f(\xx) + \gg\trans (\yy-\xx), \quad \forall \yy \in \R^d\}.
\end{equation}

\begin{definition}[Strong convexity]
    Let $\mu > 0$.
    A function $f$ is a \textit{$\mu$-strongly convex} if for any $\xx, \yy \in \R^d$ and any $\gg \in \partial f(\xx)$, we have
    \begin{equation}
        f(\yy) \ge f(\xx) + \gg^\top (\yy - \xx) + \frac{\mu}{2}\|\yy - \xx\|^2.
    \end{equation}
\end{definition}

Now, we list the variants of QHD for non-smooth optimization problems with different convexity structures:
\begin{enumerate}
    \item[(a)] QHD for strongly convex optimization problems (\textbf{QHD-SC}):
    \begin{equation}\label{eqn:qhd-sc}
        i \partial_t \Psi(t) = H_{\mathrm{SC}}(t)\Psi(t),\quad
        H_{\mathrm{SC}}(t) = e^{-2\sqrt{\mu}t} \left(-\frac{1}{2} \Delta \right) + e^{2\sqrt{\mu}t} f(x),\quad t \in [0, \infty),
    \end{equation}
    where $f$ is a $\mu$-strongly convex function. 
    
    \item[(b)] QHD for convex optimization problems (\textbf{QHD-C}):
    \begin{equation}\label{eqn:qhd-c}
        i \partial_t \Psi(t) = H_{\mathrm{C}}(t)\Psi(t),\quad H_{\mathrm{C}}(t) = \frac{1}{t^3} \left(-\frac{1}{2} \Delta \right) + t^3 f(x),\quad t \in [T_0, \infty),
    \end{equation}
    where $f$ is a convex function, $T_0 > 0$ is the starting time.\footnote{We require $T_0 > 0$ to ensure the quantum Hamiltonian $H_{\mathrm{C}}(t)$ is well-defined for all $t$.} 

    \item[(c)] QHD for non-convex optimization problems (\textbf{QHD-NC}):
    \begin{equation}\label{eqn:qhd-nc}
        i \partial_t \Psi(t) = H_{\mathrm{NC}}(t)\Psi(t),\quad
        H_{\mathrm{NC}}(t) = \frac{1}{\alpha t^{1/3}} \left(-\frac{1}{2}\Delta \right) + \alpha t^{1/3} f(x),\quad t \in [T_0, \infty),
    \end{equation}
    where $\alpha > 0$ is a parameter that controls the rate of the evolution. The choice of $\alpha$, which depends on $f$ and the target optimality gap, will be detailed in the proof of~\thm{qhd-non-convex}.
\end{enumerate}
Note that in all three cases, the QHD Hamiltonian takes the following form:
\begin{align}\label{eqn:cont-time-qhd-general}
    H(t) = \frac{1}{\lambda(t)} \left(-\frac{1}{2}\Delta \right) +\lambda(t) f(x),
\end{align}
where $\lambda(t)$ is a monotonically increasing function in time $t$. In~\algo{cont-time-qhd}, we give the quantum algorithm that implements continuous-time QHD to solve the optimization problem~\eqn{problem-form}.

We assume our quantum algorithm only has access to the \textit{quantum evaluation oracle} (i.e., zeroth-order oracle), which is defined as a unitary map $O_f$ on $\R^d\otimes \R$ such that for any $\ket{\xx} \in \R^d$,
\begin{align}
    O_f\left(\ket{\xx}\otimes\ket{0}\right) = \ket{\xx}\otimes\ket{f(\xx)}.
\end{align}
Note that the quantum evaluation oracle can be coherently accessed.

\begin{algorithm}[H]
\begin{small}
\caption{Continuous-time QHD}
\label{algo:cont-time-qhd}
\vspace{5pt}
\textbf{Classical inputs:} time-dependent parameter $\lambda(t)$ (\textbf{QHD-SC}: $\lambda(t) = e^{2\sqrt{\mu}t}$; \textbf{QHD-C}: $\lambda(t) = t^3$; \textbf{QHD-NC}: $\lambda(t) = \alpha t^{1/3}$), starting time $T_0$, stopping time $T_f$\\
\textbf{Quantum inputs:} zeroth-order oracle $O_f$, an initial guess state $\ket{\Psi_0}$ \\
\textbf{Output:} an approximate solution to the minimization problem~\eqn{problem-form}
\vspace{7pt}
\hrule
\vspace{5pt}
\begin{algorithmic}

\State 1. Initialize the quantum computer to the quantum state $\ket{\Psi_0(\xx)}$
\State 2. Simulate the quantum dynamics described by~\eqn{cont-time-qhd-general} for time $[T_0, T_f]$ using~\cor{continuous-time-qhd-complexity}.
\State 3. Measure the final quantum state $\ket{\Psi(T)}$ using the computational basis.
\end{algorithmic}
\vspace{5pt}
\end{small}
\end{algorithm}

In the last step, the computational basis measurement corresponds to measuring the wave function $\Psi(T,\xx)$ using the position quadrature observable $\hat{\xx}$. 
The outcome is a random variable $X_T$ distributed as the probability density $|\Psi(T,\xx)|^2$.

\paragraph{Complexity analysis.}
The complexity of~\algo{cont-time-qhd} is mainly determined by the cost of simulating the Schr\"odinger equation governed by the time-dependent Hamiltonian operator~\eqn{cont-time-qhd-general}. Previous results (e.g.,~\cite{leng2023quantum}) focus on the quantum simulation with a smooth potential function, which do not apply to our non-smooth setting. 
To address this issue, we prove the following theorem for simulating Schr\"odinger equation with non-smooth (but bounded) potential functions, as detailed in~\prop{spectral-method}. The proof of the result can be found in~\append{q-sim-nonsmooth-f}, which relies on a new Sobolev estimate that characterizes the growth of derivatives of the wave function $\Psi(t,x)$ in continuous-time QHD.

\begin{proposition}\label{prop:spectral-method}
    Consider the time-dependent Schr\"odinger equation:
    \begin{align}\label{eqn:schrodinger-standard-form}
        i \frac{\partial}{\partial s}\ket{\Psi(s,x)} = \left[-\frac{1}{2}\nabla^2 + \varphi(s)f(x)\right]\ket{\Psi(s,x)},
    \end{align}
    where the function $\Psi(s,x) \colon [0, s_f]\times \Omega \to \C$ is subject to a period and analytic initial condition $\Psi_0(x)$ over the box $\Omega = [-M, M]^d$ for some $M > 0$ with periodic boundary condition.
    Suppose the potential function $f(x)$ is bounded in $\Omega$. We define 
    \begin{equation}\label{eqn:xi-and-S}
        \xi = \int^{s_f}_0 \varphi(s)\d s,\quad S = \sup_{x\in \Omega} |f(x)|.
    \end{equation}
    Then, the Schr\"odinger equation \eqn{schrodinger-standard-form} can be simulated for time $s\in [0, s_f]$ up to accuracy $\eta$ with $\widetilde{\mathcal{O}}\left(\xi S\right)$
    queries to the quantum evaluation oracle $O_V$, and an additional gate complexity 
    \begin{equation}
        \widetilde{\mathcal{O}}\left(\xi S(d + \poly(\xi S))\right),
    \end{equation}
    where the $\widetilde{\mathcal{O}}(\cdot)$ notation suppresses poly-logarithmic factors in $\eta$.
\end{proposition}

To leverage the simulation algorithm as described in~\prop{spectral-method}, we need to map the QHD equation to the standard form as in~\eqn{schrodinger-standard-form}. We consider a change of variable $s \mapsto t(s)$, where $t\colon [0, s_f] \to [T_0, T_f]$ represents a time scaling (the function $t(s)$ will be determined later), and introduce a new function $\widetilde{\Psi}(s) \coloneqq \Psi(t(s))$, where $\Psi(t)$ solves the QHD equation:
\begin{equation}
    i \partial_t \Psi(t) = H(t) \Psi(t),\quad H(t) = \frac{1}{\lambda(t)} \left(-\frac{1}{2}\Delta \right) +\lambda(t) f(x).
\end{equation}
By the chain rule, we have
\begin{equation}
    i\partial_s \widetilde{\Psi}(s) = \frac{\d t}{\d s}\left[\frac{1}{\lambda(t(s))} \left(-\frac{1}{2}\Delta \right) +\lambda(t(s)) f(x) \right]\widetilde{\Psi}(s).
\end{equation}
To obtain the standard form with a Hamiltonian $-\Delta 
+ \varphi(s)f$, we require
\begin{equation}
    \frac{\d t}{\d s} = \lambda(t(s)),
\end{equation}
which is a differential equation and can be solved for each choice of $\lambda(t)$: 
\begin{enumerate}
    \item \textbf{QHD-SC}: $\lambda(t) = e^{2\sqrt{\mu}t}$ for $t \in [0, T_f]$, $t(s) = \frac{1}{2\sqrt{\mu}}\log\left(\frac{1}{1 - 2\sqrt{\mu}s}\right)$ for $s \in \left[0, \frac{1}{2\sqrt{\mu}}\left(1 - e^{-2\sqrt{\mu}T_f}\right)\right]$.
    \item \textbf{QHD-C}: $\lambda(t) = t^3$ for $t \in [T_0, T_f]$, $t(s) = \frac{1}{\sqrt{T^{-2}_0 - 2s}}$ for $s\in [0, (T^{-2}_0 - T^{-2}_f)/2]$.
    \item \textbf{QHD-NC}: $\lambda(t) = \alpha t^{1/3}$ for $t \in [T_0, T_f]$, $t(s) = (T^{2/3}_0 + \frac{2\alpha}{3}s)^{3/2}$ for $s\in \left[0, \frac{3}{2\alpha} \left(T^{2/3}_f - T^{2/3}_0\right)\right]$.
    
\end{enumerate}
Therefore, to simulate continuous-time QHD, it is equivalent to simulate the following Schr\"odinger equation
\begin{equation}\label{eqn:rescaled-cont-time-qhd}
    i \partial_s \Psi(s) = \mathcal{H}(s) \Phi(s),\quad \mathcal{H}(s) = -\frac{1}{2}\Delta +\varphi(s) f(x),\quad s \in [0, s_f],
\end{equation}
where the rescaled time-dependent function $\varphi(s)$ and the end evolution time $s_f$ is given in~\tab{rescaled-qhd-params}. Therefore, we can invoke~\prop{spectral-method} to simulate continuous-time QHD. The overall query complexity is summarized in~\cor{continuous-time-qhd-complexity}.

\begin{table}[htbp]
    \centering
        \begin{tabular}{|c|c|c|}
            \hline
            Variants of QHD & (Rescaled) time-dependent function $\varphi(s)$ & End 
            Evolution Time $s_f$ \\
            \hline
            \textbf{QHD-SC} & $\varphi(s) = (1-2\sqrt{\mu}s)^{-2}$ & $s_f  = \frac{1}{2\sqrt{\mu}}\left(1 - e^{-2\sqrt{\mu}T_f}\right)$ \\
            \hline
            \textbf{QHD-C} & $\varphi(s) = (T^{-2}_0 - 2s)^{-3}$ & $s_f =(T^{-2}_0 - T^{-2}_f)/2$ \\
            \hline
            \textbf{QHD-NC} & $\varphi(s) = \alpha^2 \left(T^{2/3}_0 + \frac{2\alpha}{3}s\right)$ & $s_f = \frac{3}{2\alpha} \left(T^{2/3}_f - T^{2/3}_0\right)$\\
            \hline
        \end{tabular}
    \caption{Time-dependent functions and corresponding end evolution time in the rescaled continuous-time QHD~\eqn{rescaled-cont-time-qhd}.}
    \label{tab:rescaled-qhd-params}
\end{table} 

\begin{corollary}\label{cor:continuous-time-qhd-complexity}
    Given an analytic initial data, the query complexity of simulating QHD is: \textit{(i)} $\widetilde{\mathcal{O}}\left(\frac{1}{\sqrt{\mu}} e^{2\sqrt{\mu}T_f}\right)$ for \textbf{QHD-SC}, \textit{(ii)} $\widetilde{\mathcal{O}}\left(T^4_f - T^4_0\right)$ for \textbf{QHD-C}, and \textit{(iii)} $\widetilde{\mathcal{O}}\left(\alpha(T^{4/3}_f - T^{4/3}_0)\right)$ for \textbf{QHD-NC}.
\end{corollary}
\begin{proof}
    By the change of variable, we have
    \begin{equation}
        \int^{s_f}_0 \varphi(s)\d s = \int^{s_f}_0 \lambda(t(s)) \frac{\d t}{\d s} \d s = \int^{T_f}_{T_0} \lambda(t) \d t.
    \end{equation}
    Then, we can compute the query complexity based on~\prop{spectral-method}.
\end{proof}

\begin{remark}
    We note that the query complexity of \textbf{QHD-SC} scales exponentially in the evolution time $T_f$. In the following section, we will show that \textbf{QHD-SC} converges to the global minimum at a linear rate, meaning that the evolution time can be $\mathcal{O}(\log(1/\epsilon))$, where $\epsilon$ is a target optimality gap. As a result, the overall query complexity of \textbf{QHD-SC} still scales polynomially in $1/\epsilon$.
\end{remark}

\subsection{Convergence analysis}

\subsubsection{Weak derivative and subgradient}\label{sec:weak-derivative}
In this section, we introduce the notion of weak derivatives for non-differentiable functions.  

\begin{definition}
    A function $f\colon \R^d \to \R$ is a \textit{locally integrable} function, denoted by $f \in L^1_{\mathrm{loc}}(\R^d)$, if 
    \begin{align}
        \int_{\R^d} |f(\xx) \varphi(\xx)|~ \d \xx < +\infty
    \end{align}
    for any test function $\varphi \in C^\infty_c(\R^d)$, i.e., $\varphi\colon \R^d \to \R$ is a compactly supported, smooth function.
\end{definition}

\begin{definition}
    Let $f \in L^1_{\mathrm{loc}}(\R^d)$ and $\gg = (g_1,\dots,g_d)\trans$ with $g_j \in L^1_{\mathrm{loc}}(\R^d)$ for all $j = 1,\dots,d$. We say that $g_j$ is a \textit{weak partial derivative} of $f$ if
    \begin{equation}
        \int f \partial_j \varphi~\d x = - \int g_j \varphi~\d x
    \end{equation}
    for all test functions $\varphi \in C^\infty_c(\R^d)$. 
    If $\gg$ exists, we say $f$ is \textit{weakly differentiable}, and $\gg$ in the \textit{weak gradient} of $f$.
\end{definition}

A weak derivative of $f$, if exists, is unique up to a set of measure zero~\cite{evans2022partial}.
Similar to strong derivatives, we have the integral by parts formula for weak derivatives, as summarized in the following lemma.

\begin{lemma}[Integral by parts]\label{lem:weak-der-c}
    Suppose that $f \in L^1_{\mathrm{loc}}(\R^d)$ admits a weak gradient $\gg$. Let $\Phi(\xx) \colon \R^d \to \R^d$ such that $\Phi_j(\xx) \in W^{1,1}(\R^d) \cap C^\infty(\R^d)$ for each $j = 1,\dots, d$. Then, we have
    \begin{align}
        \int (\nabla \cdot \Phi) f ~\d \xx = -\int \Phi \cdot \gg ~\d \xx.
    \end{align}
\end{lemma}
\begin{proof}
It suffices to prove the lemma for $d = 1$. Since $\Phi(x) \in W^{1,1}(\R)\cap C^\infty(\R)$, we can construct a sequence of test functions $\{\varphi_k\}^\infty_{k=1} \subset C^\infty_c(\R)$ such that \textit{(i)} $\varphi_k(x) \le \Phi(x)$ for all $x$, and \textit{(ii)} $\varphi_k(x) \to \Phi(x)$ for every $x$. Then, by the dominated convergence theorem, we have
\begin{align}
    \int (\partial_x \Phi)(x) f(x) \d x = \lim_{k\to\infty} \int (\partial_x  \varphi_k)(x) f(x) \d x
    = \lim_{k\to \infty} -\int  \varphi_k(x) g(x) \d x = - \int \Phi(x) g(x) \d x
\end{align}
\end{proof}

In general, even if a function $f$ is differentiable almost everywhere with a classical derivative $f'$ defined over $\R^d \setminus \Omega_f$, the classical derivative $f'$ does not necessarily coincide with the weak derivative of $f$. For example, we consider the indicator function $\chi_{(0,1)}$ defined over the interval $(0,1)$. It is straightforward to check that the weak derivative of $\chi_{(0,1)}$ is $\delta_0 - \delta_1$, i.e., the difference between two Dirac measures. However, the classical derivative of $\chi_{(0,1)}$ is $0$ except for $x \in \{0,1\}$.
In this paper, we assume $f$ is Lipschitz continuous, which implies that $f$ is differentiable almost everywhere and $f'$ automatically corresponds to a weak derivative of $f$.

\begin{theorem}
    Assume $f\colon \R^d \to \R$ is locally Lipschitz. Then, $f$ is differentiable almost everywhere, and its gradient equals its weak gradient almost everywhere.
\end{theorem}
\begin{proof}
    This is a direct consequence of Theorem 4 and Theorem 5 in Section 5 of \cite{evans2022partial}.
\end{proof}

\begin{corollary}\label{cor:lipschitz-subgradient}
    Assume $f\colon \R^d \to \R$ is locally Lipschitz. If $\gg$ is a subgradient of $f$, then $\gg$ equals the weak derivative of $f$ almost everywhere.
\end{corollary}
\begin{proof}
    This follows from the fact that the subgradient $\gg$ must coincide with the strong gradient $\nabla f$ wherever it exists. 
\end{proof}

\subsubsection{Case I: non-smooth strongly convex optimization}
The main convergence result for \textbf{QHD-SC} is summarized as follows.

\begin{theorem}\label{thm:qhd-strongly-convex}
     Suppose that $f(x)$ is locally integrable, Lipschitz continuous, $\mu$-strongly convex, and potentially non-smooth. Let $\Psi(t,\xx)$ be the solution to QHD-SC~\eqn{qhd-sc} subject to a smooth initial state $\Psi_0(\xx)$, and $X_t$ be a random variable following the distribution $|\Psi(t,\xx)|^2$. Then, we have
    \begin{align}
        \mathbb{E}[f(X_t)] - f(\xx^*) \le O(e^{-\sqrt{\mu}t}).
    \end{align}
\end{theorem}

\begin{proof}
    We prove~\thm{qhd-strongly-convex} using a Lyapunov function argument. For $t\ge 0$, let $\Psi(t,\xx)$ be the solution to \textbf{QHD-SC}, i.e., 
\begin{equation}
    i \partial_t \Psi(t) = H_{\mathrm{SC}}(t)\Psi(t),\quad
    H_{\mathrm{SC}}(t) = e^{-2\sqrt{\mu}t} \left(-\frac{1}{2} \Delta \right) + e^{2\sqrt{\mu}t} f(x),
\end{equation}
subject to an initial state $\Psi(0,\xx) = \Psi_0(\xx)$.
Without loss of generality, we assume $f$ has a unique global minimum $f(\xx^*) = 0$ with $\xx^* = 0$. We consider the following Lyapunov function:
\begin{align}\label{eqn:lyapunov-f-strongly-convex}
    \mathcal{E}(t) = \left\langle \hat{O}_{\mathrm{SC}}(t)\right\rangle_t,\quad \hat{O}_{\mathrm{SC}}(t) \coloneqq f + \frac{1}{4}e^{-4\sqrt{\mu}t}\sum^d_{j=1}p^2_j + \frac{1}{4}\sum^d_{j=1}J^2_j,\quad J_j = e^{-2\sqrt{\mu}t}p_j + 2\sqrt{\mu}x_j.
\end{align}
Here, the operators $p_j$ and $x_j$ are the momentum and position operators for the $j$-th coordinate, as described in~\eqn{position-momentum-ops}.
As shown in~\lem{lyapunov-estimate-sc}, the Lyapunov function $\mathcal{E}(t)$ is strictly non-increasing in time: 
\begin{align}
    \dot{\mathcal{E}}(t) \le -\frac{\sqrt{\mu}}{4}\mathcal{E}(t) < 0.
\end{align}
Moreover, by Gr\"onwall's inequality, for any $t \ge 0$, we have 
\begin{equation}\label{eqn:proof-sc-1}
    \mathcal{E}(t) \le e^{-\frac{\sqrt{\mu}}{4}t}\mathcal{E}(0)
\end{equation}
for any $t \ge 0$. 
Meanwhile, note that the operators $p^2_j$ and $J^2_j$ are positive for each $j = 1,\dots, d$. By the definition of the Lyapunov function, we have 
\begin{equation}\label{eqn:proof-sc-2}
    \mathcal{E}(t) = \langle f \rangle_t + \frac{1}{4}e^{-4\sqrt{\mu}t}\sum^d_{j=1}\langle p^2_j\rangle_t + \frac{1}{4}\sum^d_{j=1}\langle J^2_j\rangle_t \ge \langle f \rangle_t.
\end{equation}
Combining~\eqn{proof-sc-1} and \eqn{proof-sc-2}, we show the exponential convergence of \textbf{QHD-SC}:
\begin{align}
    \langle f \rangle_t \le e^{-\frac{\sqrt{\mu}}{4}t}\mathcal{E}(0),
\end{align}
which concludes the proof of~\thm{qhd-strongly-convex}.
\end{proof}

\begin{lemma}\label{lem:lyapunov-estimate-sc}
Let $f$ be a locally integrable and Lipschitz continuous function. For $t \ge 0$, we denote $\Psi(t,\xx)$ as the solution to \textbf{QHD-SC} with an initial state $\Psi_0(\xx) \in C^\infty(\R^d)\cap L^2(\R^d)$. Let the function $\mathcal{E}(t)$ be the same as in~\eqn{lyapunov-f-strongly-convex}. Then, for any $t \ge 0$, $\Psi(t,\xx)$ is a smooth function in $\xx$ and
\begin{align}
    \dot{\mathcal{E}}(t) = \sqrt{\mu}\left[-e^{-4\sqrt{\mu}t}\sum^d_{j=1}\langle p^2_j\rangle_t + \int \nabla \cdot \left(|\Psi(t,\xx)|^2\xx\right) f(\xx) \d \xx\right].
\end{align}
Moreover, if $f$ is $\mu$-strongly convex with $\xx^* = 0$ and $f(\xx^*) = 0$, the following holds for all $t \ge 0$:
\begin{align}
    \dot{\mathcal{E}}(t) \le -\frac{\sqrt{\mu}}{4}\mathcal{E}(t).
\end{align}
\end{lemma}
\begin{proof}
    The smoothness of the solution at any $t$ follows from the standard regularity theory of Schr\"odinger equations, see~\thm{regularity}.
    Similar to in the proof of~\lem{e-dot-c}, the time derivative of the Lyapunov function can be expressed as follows:
    \begin{align}
        \partial_t \mathcal{E}(t) = \langle \partial_t\hat{O}_{\mathrm{SC}}\rangle_t + \langle i[H_{\mathrm{SC}}(t), \hat{O}_{\mathrm{SC}}]\rangle_t,
    \end{align}
    with
    \begin{align}
        \langle \partial_t \hat{O}_{\mathrm{SC}}(t)\rangle_t = \left\langle -2\sqrt{\mu}e^{-4\sqrt{\mu}t}\sum^d_{j=1}p^2_j - \mu e^{-2\sqrt{\mu}t}\sum^d_{j=1}\{x_j, p_j\} \right\rangle_t,
    \end{align}
    and 
    \begin{align}
        \langle i[H_{\mathrm{SC}}(t), \hat{O}_{\mathrm{SC}}(t)]\rangle_t 
        &= \frac{\mu}{2} e^{-2\sqrt{\mu}t}\sum^d_{j=1}\langle i[p^2_j,x^2_j]\rangle_t 
        + \frac{\sqrt{\mu}}{4} e^{-4\sqrt{\mu}t}\sum^d_{j=1}\langle i[p^2_j, \{x_j, p_j\}] \rangle_t\\
        &+ \sqrt{\mu} \int \nabla \cdot \left(|\Psi|^2 \xx\right) f(\xx) \d \xx.
    \end{align}
    For any $j = 1,\dots,d$, we have the following commutation relations~\cite[Appendix B.1, Lemma~2]{leng2023qhd}:
    \begin{align}
        i[p^2_j,x^2_j] = 2\{x_j,p_j\},\quad i[p^2_j, \{x_j, p_j\}] = 4p^2_j.  
    \end{align} 
    By adding the two parts of $\partial_t \mathcal{E}(t)$ together and using the above commutation relations, we obtain
    \begin{align}\label{eqn:strongly-eq-1}
        \partial_t \mathcal{E}(t) = \sqrt{\mu}\left(-e^{-4\sqrt{\mu}t}\sum^d_{j=1}\langle p^2_j \rangle_t + \int \nabla \cdot \left(|\Psi(t,\xx)|^2\xx\right) f(\xx)\d \xx\right),
    \end{align}
    which proves the first part of the lemma. 
    When $f$ is Lipschitz continuous and $\mu$-strongly convex, there is a subgradient $\gg$ that coincides with its weak derivative almost everywhere. By~\lem{weak-der-c}, we have
    \begin{align}\label{eqn:strongly-eq-2}
        \int \nabla \cdot \left(|\Psi|^2 \xx\right)f(\xx) \d \xx = - \langle \xx \cdot \gg(\xx) \rangle_t \le - \left\langle f + \frac{\mu}{2}\|\xx\|^2\right\rangle_t.
    \end{align}
    Therefore, by plugging~\eqn{strongly-eq-2} into~\eqn{strongly-eq-1}, we have
    \begin{align}
        \partial_t \mathcal{E}(t) \le -\sqrt{\mu} \left\langle e^{-4\sqrt{\mu}t}\sum^d_{j=1}p^2_j
        + f + \frac{\mu}{2} \sum^d_{j=1}x^2_j\right\rangle_t = \mathcal{J}_1(t) + \mathcal{J}_2(t),
    \end{align}
    where 
    \begin{align}
        \mathcal{J}_1(t) \coloneqq -\sqrt{\mu}\left\langle\frac{1}{8}e^{-4\sqrt{\mu}t}\sum^d_{j=1}p^2_j + \frac{\mu}{2}\sum^d_{j=1}x^2_j\right\rangle_t,\quad \mathcal{J}_2(t) \coloneqq - \sqrt{\mu}\left\langle f + \frac{7}{8}e^{-4\sqrt{\mu}t}\sum^d_{j=1}p^2_j\right\rangle_t.
    \end{align}
    Furthermore, by invoking the operator inequality $\frac{1}{2}(A+B)^2 \le A^2 + B^2$ (see~\eqn{J1}) and the positivity of $f$ and $p^2_j$ (see~\eqn{J2}), we obtain that
    \begin{align}
        \mathcal{J}_1(t) &= -\frac{\sqrt{\mu}}{8}\sum^d_{j=1}\left\langle e^{-4\sqrt{\mu}t}p^2_j + 4\mu x^2_j\right\rangle_t \le -\frac{\sqrt{\mu}}{16}\sum^d_{j=1}J^2_j,\label{eqn:J1}\\
        \mathcal{J}_2(t) &= -\frac{\sqrt{\mu}}{4}\mathcal{E}(t) + \frac{\sqrt{\mu}}{16}\left(\sum^d_{j=1}J^2_j\right) - \left(\frac{3\sqrt{\mu}}{4}\langle f\rangle_t + \frac{13\sqrt{\mu}}{16}e^{-4\sqrt{\mu}t}\sum^d_{j=1}\langle p^2_j\rangle_t \right)\\
        &\le -\frac{\sqrt{\mu}}{4}\mathcal{E}(t) + \frac{\sqrt{\mu}}{16}\left(\sum^d_{j=1}J^2_j\right).\label{eqn:J2}
    \end{align}
    As a consequence, we have
    \begin{align}
        \dot{\mathcal{E}}(t) \le \mathcal{J}_1(t) + \mathcal{J}_2(t) \le - \frac{\sqrt{\mu}}{4}\mathcal{E}(t),
    \end{align}
    which proves the second part of the lemma.
\end{proof}

\subsubsection{Case II: non-smooth convex optimization}
The main convergence result for \textbf{QHD-C} is summarized as follows.

\begin{theorem}\label{thm:qhd-convex}
    Suppose that $f(x)$ is locally integrable, Lipschitz continuous, convex, and potentially non-smooth. Let $\Psi(t,\xx)$ be the solution to \textbf{QHD-C}~\eqn{qhd-c} subject to a smooth initial state $\Psi_0(\xx)$, and $X_t$ be a random variable following the distribution $|\Psi(t)|^2$. Then, 
    \begin{align}
        \mathbb{E}[f(X_t)] - f(\xx^*) \le O(t^{-2}).
    \end{align}
\end{theorem}

\begin{proof}
Let $\Psi(t,\xx)$ be the solution to the differential equation of \textbf{QHD-C} for $t \ge T_0 > 0$ subject to an initial state $\Psi(0,\xx) = \Psi_0(\xx)$.
For a self-adjoint operator $\hat{O}$, we define the expected value of $\hat{O}$ as a function of time $t$:
\begin{align}
    \langle \hat{O} \rangle_t \coloneqq \int \overline{\Psi(t,\xx)} \hat{O}\Psi(t,x)\d \xx = \braket{\Psi(t)}{\hat{O}\Big|\Psi(t)}. 
\end{align}
With this notation, we have $\mathbb{E}[f(X_t)] = \langle f\rangle_t$.

Let $f$ be a convex objective function. Without loss of generality, we assume $f(\xx^*)=0$ and $\xx^* = 0$. We consider the following Lyapunov function,
\begin{align}\label{eqn:lyapunov-f-convex}
    \mathcal{E}(t) = \left\langle \hat{O}_\mathrm{C}(t) \right\rangle_t, \quad \hat{O}_\mathrm{C}(t) \coloneqq  t^2 f + \frac{1}{2}\sum^d_{j=1}\left(\frac{1}{t^4}p^2_j + \frac{2}{t^2}\{p_j,x_j\} + 4x^2_j\right).
\end{align}
Here, the operators $p_j$ and $x_j$ denote the momentum and position operators for the $j$-th coordinate. They should be interpreted as operators acting on a smooth test function $\varphi(\xx)\colon \R^d \to \R$ as follows:
\begin{align}\label{eqn:position-momentum-ops}
    (p_j \varphi)(\xx) \coloneqq -i\partial_{x_j}\varphi(\xx),\quad (x_j \varphi)(\xx) \coloneqq x_j \varphi(\xx).
\end{align}
$\{A,B\} = AB+BA$ represents the anti-commutator between operators $A$ and $B$. 

For each $j = 1,\dots, d$, we observe that 
\begin{align}
    t^{-4}p^2_j + 2t^{-2}\{x_j, p_j\} + 4 x^2_j = \left(t^{-2}p_j + 2x_j\right)^2 \ge 0,
\end{align}
which immediately implies that $\hat{O}_\mathrm{C}(t) - t^2f$ is a positive operator, i.e., 
\begin{align}
    t^2\langle f \rangle_t \le \mathcal{E}(t).
\end{align}
In~\lem{e-dot-c}, we show that $\mathcal{E}(t)$ is a non-increasing function in time for all $t \ge T_0$. As a result,
\begin{align}
    \langle f \rangle_t \le \frac{\mathcal{E}(t)}{t^2} \le \frac{\mathcal{E}(T_0)}{t^2},
\end{align}
which proves~\thm{qhd-convex}.
\end{proof}

\begin{lemma}\label{lem:e-dot-c}
Let $f$ be a locally integrable and Lipschitz continuous function. For $t \ge T_0$, we denote $\Psi(t,\xx)$ as the solution to \textbf{QHD-C} with an initial state $\Psi_0(\xx) \in  C^\infty(\R^d)\cap L^2(\R^d)$. Let the function $\mathcal{E}(t)$ be the same as in~\eqn{lyapunov-f-convex}. Then, for any $t \ge T_0$, $\Psi(t,\xx)$ is a smooth function in $\xx$ and
\begin{align}
    \dot{\mathcal{E}}(t) = 2t \left[\langle f \rangle_t + \int \nabla \cdot \left(|\Psi(t,\xx)|^2 \xx\right) f(\xx)\d \xx \right].
\end{align}
Moreover, if $f$ is a convex function with $\xx^*=0$ and $f(\xx^*) = 0$, we have $\dot{\mathcal{E}}(t)  \le 0$
for all $t \ge T_0$.
\end{lemma}
\begin{proof}
    The smoothness of the solution at any $t$ follows from the standard regularity theory of Schr\"odinger equations, see~\thm{regularity}.
    Recall that we have $\mathcal{E}(t) = \braket{ \Psi(t)}{\hat{O}_\mathrm{C}\Big|\Psi(t)}$, with $\hat{O}_\mathrm{C} = t^2 f + \frac{1}{2}\sum^d_{j=1}\left(\frac{1}{t^4}p^2_j + \frac{2}{t^2}\{p_j,x_j\} + 4x^2_j \right)$.
    By the distribution law of derivative, we have
    \begin{align}
        \partial_t \mathcal{E}(t) &=\braket{\partial_t \Psi(t)}{\hat{O}_\mathrm{C}\Big|\Psi(t)} + \braket{\Psi(t)}{\partial_t \hat{O}_\mathrm{C}\Big|\Psi(t)} +\braket{\Psi(t)}{\hat{O}_\mathrm{C}\Big|\partial_t\Psi(t)}\\
        &= \langle \partial_t\hat{O}_\mathrm{C}(t)\rangle_t + \langle i[H_{\mathrm{C}}(t), \hat{O}_\mathrm{C}(t)]\rangle_t,
    \end{align}
    where the last step follow from the equation $\partial_t \Psi = -iH_{\mathrm{C}}(t)\Psi$. Here, $[\cdot,\cdot]$ is the commutator of two operators: $[A, B] = AB - BA$.
    Further calculation shows
    \begin{align}\label{eqn:dEdt_part1}
        \langle \partial_t\hat{O}_\mathrm{C}(t)\rangle_t = \left\langle 2tf -2t^{-3}\sum^d_{j=1}\{x_j,p_j\} - 2t^{-5} \sum^d_{j=1}p^2_j\right\rangle_t,
    \end{align}
    and 
    \begin{align}\label{eqn:dEdt_part2}
        \langle i[H_{\mathrm{C}}(t), \hat{O}_\mathrm{C}(t)]\rangle_t = \left\langle t^{-3}\sum^d_{j=1}i[p^2_j,x^2_j] +t^{-5}\sum^d_{j=1}i[p^2_j,\{x_j,p_j\}/2]\right\rangle_t + 2t\int \nabla \cdot \left(|\Psi(t)|^2\xx\right)f(\xx)\d \xx.
    \end{align}
    Note that the last term in~\eqn{dEdt_part2} is equivalent to $-2t \langle \nabla f(\xx) \cdot \xx\rangle_t$ if $f$ is continuously differentiable. However, since $f$ is non-smooth, we need to move the derivatives to the ``smooth part'' of the integral in order to make the expression well-defined. 
    For any $j = 1,\dots,d$, we have the following commutation relations~\cite[Appendix B.1, Lemma~2]{leng2023qhd}:
    \begin{align}
        i[p^2_j,x^2_j] = 2\{x_j,p_j\},\quad i[p^2_j, \{x_j, p_j\}] = 4p^2_j.  
    \end{align} 
    Therefore, by adding up~\eqn{dEdt_part1} and \eqn{dEdt_part2}, we end up with
    \begin{align}\label{eqn:convex-eq-2}
        \partial_t \mathcal{E}(t) = 2t\langle f\rangle_t + 2t\int \nabla \cdot \left(|\Psi(t)|^2\xx\right)f(\xx)\d \xx,
    \end{align}
    which proves the first part of the lemma.
    When $f$ is Lipschitz continuous, \cor{lipschitz-subgradient} implies that there exists a subgradient fo $f$ (denoted by $\gg$) that equals the weak derivative of $f$ almost everywhere.
    By \lem{weak-der-c}, we have
    \begin{align}\label{eqn:convex-eq-1}
        \int \nabla \cdot \left(|\Psi(t,\xx)|^2\xx\right)f(\xx)\d \xx = -\int |\Psi(t,\xx)|^2 \xx \cdot \gg(\xx) \d \xx = -\langle \xx \cdot \gg \rangle_t.
    \end{align}
    Plugging \eqn{convex-eq-1} into~\eqn{convex-eq-2}, we have
    \begin{align}
        \dot{\mathcal{E}}(t) = 2t\langle f - \xx \cdot g\rangle_t \le 0, 
    \end{align}
    where the last inequality follows from the fact that $\gg$ is a subgradient so $f + \gg \cdot (\xx^* - \xx) \le f(\xx^*) = 0$.
\end{proof}

\subsubsection{Case III: non-smooth non-convex optimization}
In this section, we prove the global convergence of QHD for non-convex problems.
For simplicity, we impose box constraints on the optimization problem
\begin{equation}
    \min_{\xx \in \Omega} f(\xx),
\end{equation}
where $\Omega \coloneqq [-1,1]^d$.
We note that this assumption is made for technical reasons and does not materially affect the behavior of the quantum algorithm.
Moreover, the bounded search space can be further mapped into $\Omega$ through appropriate variable transformations.
Since the domain is no longer $\mathbb{R}^d$, we make the QHD wave function $\Psi(t)$ in \eqn{qhd-nc} satisfy a Dirichlet boundary condition, namely
\begin{equation}
    \Psi(t,\xx) = 0, \quad \text{for $\xx \in \partial \Omega$}.
\end{equation}

Define $\ket{\phi}$ as the ground state of $-\Delta$, where $\Delta$ is the Laplacian operator.
Its wave function can be written analytically:
\begin{equation}
    \phi(\xx) = \phi(x_1,\dots,x_d) = \prod_{1 \leq j \leq d} \cos (\frac{\pi}{2} x_j).
\end{equation}
Arguably, $\ket{\phi}$ is easy to prepare since it is a product state with simple components.
Now we are ready to state the main convergence result for \textbf{QHD-NC}.

\begin{theorem}\label{thm:qhd-non-convex}
    For any $\delta > 0$, there exist $\alpha = \alpha(f,\delta) > 0$ and $T_0 = T_0(f,\delta) > 0$ for \textbf{QHD-NC}~\eqn{qhd-nc} with the initial state being $\ket{\phi}$, such that the final solution $\ket{\psi} \coloneqq \Psi(T)$ satisfies $\braket{\psi}{f|\psi} \le \min f(\xx) + \delta$, where $T = T(f,\delta)$. In other words, let $X_T$ be a random variable following the distribution $|\Psi(T)|^2$, we have
    \begin{align}
        \mathbb{E}[f(X_T)] \le \min f(\xx) + \delta.
    \end{align}
\end{theorem}

The proof of \thm{qhd-non-convex} heavily relies on the analysis of spectra of \sdg\ operators $H = - \Delta + V$ where $V\colon \Omega\to\mathbb{R}$ is a continuous potential function.
Refer to \sec{sdg-spec} for important properties of \sdg\ operators.
Let us introduce necessary tools before proceeding to the formal proof.

\paragraph{Adiabatic evolution of unbounded quantum Hamiltonian.}
Consider the dynamics described by the \sdg\ equation
\begin{equation} \label{eqn:adiabatic}
    i \epsilon \frac{\d}{\d t}\ket{\psi^\epsilon(t)} = H(t)\ket{\psi^\epsilon(t)},
\end{equation}
subject to an initial state $\ket{\psi^\epsilon(0)}$.
Denote by $U^\epsilon(t)$ the propagator of the dynamics described by \eqn{adiabatic}, so the solution at time $t$ is given by
\begin{align}
    \ket{\psi^\epsilon(t)} = U^\epsilon(t) \ket{\psi^\epsilon(0)}.
\end{align}
The parameter $\epsilon > 0$ controls the time scale on which the quantum Hamiltonian $H(t)$ varies. For a small $\epsilon$, the system evolves slowly and the dynamics are relatively simple: if the system begins with an eigenstate of $H(0)$, it remains close to an eigenstate of $H(t)$. This process is called \textit{adiabatic quantum evolution}.

Formally, let $P(t)$ be the rank-1 projector onto the ground state of $H(t)$.
The quantum adiabatic theorem states that the propagator $U^\epsilon(t)$ approximately maps the ground state of $H(0)$ to that of $H(t)$ for sufficiently small $\epsilon$. 
Various formulations of the quantum adiabatic theorem exist in the literature, while most of them assume the Hamiltonian is a bounded linear operator. In QHD, however, the Hamiltonian is an unbounded operator defined on $\Omega$. Here, we introduce a quantum adiabatic theorem for unbounded Hamiltonian, which allows us to have a neat analysis without discretizing the unbounded operator in QHD.

\begin{theorem}[{\cite[Sec. 1.1]{teufel}}] \label{thm:adiabatic}
    Let $\dot{H}(t)$ and $\ddot{H}(t)$ be bounded.\footnote{i.e., they have finite \emph{graph norm}, which can be bounded by their usual $2$-norm; See~\cite{teufel}.}
    Assume that the spectral gap of $H(t)$ is always bounded above some $c > 0$.
    Then, there exists a constant $C$ such that
    \begin{equation}
        \norm{(I - P(t)) U^\epsilon(t) P(0)} \leq C \epsilon (1 + t).
    \end{equation}
\end{theorem}

\paragraph{Characterization of the low-energy spectrum.}
Our proof idea is basically relating the dynamics in \eqn{adiabatic} to that of \textbf{QHD-NC}~\eqn{qhd-nc}.
Since the quantum adiabatic theorem (\thm{adiabatic}) ensures that the system in \eqn{adiabatic} will approximately stay in the instantaneous ground state of $H(t)$ throughout, it is thus important to characterize the ground state.
This ultimately reduces to analyzing the spectral properties of $H = -h^2 \Delta + V$ in the limit $h \to 0$, which is called \emph{semiclassical analysis} in the mathematics and physics literature.
A seminal result from semiclassical analysis is the \emph{Weyl's law}, which asymptotically characterizes the low-energy spectrum of the \sdg\ operator when $h \to 0$.
\begin{definition}
    Define $N(E,H)$ as the number of eigenvalues no greater than $E$, counting multiplicities, of a self-adjoint operator $H$.
\end{definition}

\begin{lemma}[Weyl's law~\cite{ivrii2016100,rozenblum1976distribution}] \label{lem:weyl}
    Let $\Omega \subseteq \mathbb{R}^d$ be open and bounded.
    Fix a continuous\footnote{The continuity assumption can be weakened to $V \in L^{d/2}(\Omega)$; See~\cite{rozenblum1976distribution}.} function $V\colon \Omega \to \mathbb{R}$.
    Define $H(h) \coloneqq -h^2 \Delta + V$ which acts on $L^2$-functions that vanish on $\partial \Omega$.
    We have
    \begin{equation}
        N(E, H(h)) = \left(  \frac{1}{(4\pi)^{d/2}\Gamma(d/2+1)} \cdot h^{-d} \int_{\Omega} \max\{E-V(\xx),0\}^{d/2} \d \xx  \right)(1 \pm o(1)),\footnote{$A = B(1 \pm o(1))$ is equivalent to $A/B \to 1$.}
    \end{equation}
    when $E \to +\infty$ \emph{or} $h \to +0$, where $\Gamma(\cdot)$ is the gamma function.
\end{lemma}

Below we prove the main convergence result for {\bf QHD-NC}~\eqn{qhd-nc}.

\begin{proof}[Proof of~{\thm{qhd-non-convex}}]
    Let $\Omega \coloneqq [-1,1]^d$ be the domain of $f$.
    Without loss of generality assume $\inf_{x \in \Omega} f(\xx) = 0$.
    Set $H(t) \coloneqq - \Delta + tf(\xx)$ for $t \in [0,t_{\rm max}]$ with $t_{\rm max}$ to be determined later.
    Let $\norm{f}$ be the operator norm of $f$, or equivalently $\norm{f} \coloneqq \sup_{\xx \in \Omega}\{f(\xx)\}$.
    Note that $\norm{f} < \infty$ because $f$ is continuous on its compact domain $\Omega$.
    $\dot{H}(t)$ and $\ddot{H}(t)$ are bounded since $f(\xx)$ is bounded.
    Applying \lem{sdg-spec} with \rem{sdg-spec} we know that for any $t \in [0,t_{\rm max}]$, $H(t)$ has a positive spectral gap $g(t)$.
    Moreover, since $g\colon [0,t_{\rm max}] \to \mathbb{R}_+$ is obviously continuous and $[0,t_{\rm max}]$ is compact, $g(\cdot)$ has a \emph{minimum} over $[0,t_{\rm max}]$ and hence $H(t)$ satisfies the spectral gap condition in \thm{adiabatic}.
    Now, by \thm{adiabatic}, there exists a constant $C$ such that 
    \begin{equation} \label{eqn:non-convex.adiabatic}
        \norm{(1 - P(t)) U^\epsilon(t) P(0)} \leq C \epsilon (1 + t), 
    \end{equation}
    where we recall that $U^\epsilon(t)$ is the propagator of the dynamics in
    \begin{equation} \label{eqn:non-convex.eps-dynamics}
        i \epsilon \frac{\d}{\d t} \ket{\psi^\epsilon(t)} = H(t) \ket{\psi^\epsilon(t)}
    \end{equation}
    such that $ \ket{\psi^\epsilon(t)} = U^\epsilon(t) \ket{\psi^\epsilon(0)}$ for any initial state $\ket{\psi^\epsilon(0)}$, and $P(t)$ is the rank-1 projector onto the ground state of $H(t)$, denoted by $\ket{\phi_0(t)}$.
    Note that $\ket{\phi} = P(0) \ket{\phi}$ as $\ket{\phi} = \ket{\phi_0(0)}$ is the ground state of $H(0)= -\Delta$. 
    Therefore we have
    \begin{equation}
        \norm{ (I - P(t)) U^\epsilon(t) \ket{\phi}} = \norm{ (I - P(t)) U^\epsilon(t) P(0) \ket{\phi}} \leq C \epsilon (1 + t),
    \end{equation}
    where the inequality follows from \eqn{non-convex.adiabatic}.
    Now, we set
    \begin{equation} \label{eqn:non-convex.init}
        \ket{\psi^\epsilon(0)} = \ket{\phi},
    \end{equation}
    and decompose $\ket{\psi^\epsilon(t)}$ into two orthogonal components being its projection onto $P(t)$ and the remaining part:
    \begin{equation} \label{eqn:non-convex.predecomp}
        \ket{\psi^\epsilon(t)} = U^\epsilon(t) \ket{\phi} = P(t) U^\epsilon(t) \ket{\phi} + (I - P(t)) U^\epsilon(t) \ket{\phi}.
    \end{equation}
    Since $U^\epsilon(t)$ is unitary and $P(t)$ projects onto the instantaneous ground state $\ket{\phi_0(t)}$, we can rewrite \eqn{non-convex.predecomp} as
    \begin{equation} \label{eqn:non-convex.decomp}
        \ket{\psi^\epsilon(t)} = \alpha(t) \ket{\phi_0(t)} + \beta(t) \ket{r^\epsilon(t)},
    \end{equation}
    where $\beta(t) \ket{r^\epsilon(t)} \coloneqq (I - P(t)) U^\epsilon(t) \ket{\phi}$ with $\ket{r^\epsilon(t)}$ normalized, $\abs{\beta(t)} \leq C \epsilon (1 + t)$ and $\abs{\alpha(t)} = 1 - \abs{\beta(t)}^2$. 
    We will use this decomposition of $\ket{\psi^\epsilon(t)}$ to bound its expectation on $f(\cdot)$, i.e., $\braket{\psi^\epsilon(t) | f}{\psi^\epsilon(t)}$.

    Define $S(t) \coloneqq \frac{1}{t} H(t) = - \left( \frac{1}{\sqrt{t}} \right)^2  \Delta  +  f(\xx)$.
    Obviously $\ket{\phi_0(t)}$ is also the ground state of $S(t)$.
    Recall that we assumed $\inf_{x \in \Omega} f(\xx) = 0$ in the beginning of the proof.
    Thus the continuity of $f(\cdot)$ implies that for any $E > 0$,
    \begin{equation}
        \int_{\Omega} \max\{E-f(\xx),0\}^{d/2} \d \xx > 0.
    \end{equation}
    Therefore we can estimate $N(\delta/10, S(t))$ by plugging into Weyl's law $h \coloneqq 1/\sqrt{t}$ and $E \coloneqq \delta/10$:
    \begin{equation}
        N(\delta/10, S(t)) = \left(  \frac{1}{(4\pi)^{d/2}\Gamma(d/2+1)} \cdot t^{d/2} \int_{\Omega} \max\{\delta/10-f(\xx),0\}^{d/2} \d \xx  \right)(1 \pm o(1)),
    \end{equation}
    which implies $\lim_{t \to +\infty} N(\delta/10, S(t)) \to \infty$.
    It follows that the ground energy of $S(t)$ will be eventually smaller than $\delta/10$ as $t$ increases. 
    Therefore, there exists a sufficiently large $t_{\rm max}$ such that
    \begin{align} \label{eqn:non-convex.concentrate}
        \braket{\phi_0(t_{\rm max})}{f|\phi_0(t_{\rm max})} \leq \braket{\phi_0(t_{\rm max})}{S(t_{\rm max})|\phi_0(t_{\rm max})} \leq \delta/10,
    \end{align}
    where the first inequality follows from $-\Delta$ being positive semidefinite.
    Write $\alpha \coloneqq \alpha(t_{\rm max})$ and $\beta \coloneqq \beta(t_{\rm max})$.
    We have
    \begin{subequations}
    \begin{align}
        &\phantom{=}\ \braket{\psi^\epsilon(t_{\rm max})}{f |\psi^\epsilon(t_{\rm max}) } \\
        &= \abs{\alpha}^2 \braket{\phi_0(t_{\rm max})}{f | \phi_0(t_{\rm max})} + \abs{\beta}^2 \braket{r^\epsilon(t_{\rm max})}{f | r^\epsilon(t_{\rm max})}  \nonumber \\
        &\phantom{=}\ + 2\Re{\alpha \beta \braket{ \phi_0(t_{\rm max})}{f | r^\epsilon(t_{\rm max})}}  &  \text{(by \eqn{non-convex.decomp})} \\
        & \leq \abs{\alpha}^2 \delta/10 + \abs{\beta}^2\norm{f} + 2  \abs{\alpha \beta}\norm{f} & \text{(by \eqn{non-convex.concentrate})} \\
        & \leq \abs{\alpha}^2 \delta/10 + C \epsilon (1+t_{\rm max})\norm{f} + 2  \abs{\alpha }C \epsilon (1+t_{\rm max})\norm{f} & (\abs{\beta}^2 \leq \abs{\beta} \leq C \epsilon (1+t_{\rm max})) \\
        & \leq \delta/10 + C' \epsilon (1+t_{\rm max})\norm{f}, & (\abs{\alpha} \leq 1)
    \end{align}
    \end{subequations}
    where $C'$ is some constant that depends on $C$ and we assume $C \epsilon (1 + t_{\rm max}) \ll 0$.
    Let $\epsilon \coloneqq  \delta/ (10 C' (1+t_{\rm max}) \norm{f})$ we obtain 
    \begin{equation} \label{eqn:non-convex.final}
        \braket{\psi^\epsilon(t_{\rm max})}{f |\psi^\epsilon(t_{\rm max}) } \leq 2\delta/10.
    \end{equation}\

    Now we relate the dynamics of~\eqn{non-convex.eps-dynamics} to QHD via a time dilation trick.
    Let
    \begin{equation}
        H_{\rm NC}(\tau) \coloneqq  \frac{1}{\sqrt{t/2}} \left( -\frac12 \Delta\right) + \sqrt{t/2} f(\xx),\quad t\coloneqq \frac{1}{2} \left( 3 \epsilon \tau \right)^{2/3}.
    \end{equation}
    It is straightforward that $\frac{\d} {\d \tau} = \epsilon \frac{1}{\sqrt{2t}} \frac{\d}{\d t}$.
    Hence \eqn{non-convex.eps-dynamics} gives
    \begin{equation} \label{eqn:non-convex.qhd-dynamics}
        i \frac{\d} {\d \tau} \ket{\psi^\epsilon(\tau)} = i \epsilon \frac{1}{\sqrt{2t}}\frac{\d}{\d t} \ket{\psi^\epsilon(t)} = \frac{1}{\sqrt{2t}} H(t) \ket{\psi^\epsilon(t)} = H_{\rm NC}(\tau) \ket{\psi^\epsilon(\tau)},
    \end{equation}
    where $\ket{\psi^\epsilon(\tau)} \coloneqq \ket{\psi^\epsilon(t)}$ with $t = \frac{1}{2} \left( 3 \epsilon \tau \right)^{2/3}$.
    As the result, for any $t_0>0$, running QHD according to the dynamics~\eqn{non-convex.qhd-dynamics} with the initial state $\ket{\psi^\epsilon(t_0)}$ for $\tau \in [T_0, T]$ will yield the final state $\ket{\psi^\epsilon(t_{\rm max})}$, where $T \coloneqq \frac{2\sqrt{2}}{3} \frac{t_{\rm max}^{3/2}}{\epsilon}$ and $T_0 \coloneqq \frac{2\sqrt{2}}{3} \frac{t_{0}^{3/2}}{\epsilon}$.
    Note that for sufficiently small $t_0=t_0(f,\delta)$ we can always have that $\ket{\psi^\epsilon(t_0)}$ is $(\delta/5)$-close to $\ket{\psi^\epsilon(0)} = \ket{\phi}$ in $\ell_2$ distance due to the continuity of $\ket{\psi^\epsilon(\cdot)}$ from \eqn{non-convex.eps-dynamics} and \eqn{non-convex.init}.
    Consequently, replacing the initial condition $\ket{\psi^\epsilon(0)} = \ket{\phi}$ by
    \begin{equation}
        \ket{\psi^\epsilon(t_0)} = \ket{\phi},
    \end{equation}
    we have
    \begin{equation}
        \braket{\psi^\epsilon(t_{\rm max})}{f |\psi^\epsilon(t_{\rm max}) } \leq 2\delta/10 + \delta/5 + \delta/5 < \delta,
    \end{equation}
    by using triangle inequalities with \eqn{non-convex.final}, as desired in the theorem statement.
\end{proof}

\section{Discrete-time QHD for Non-Smooth Optimization}\label{sec:discrete-time-qhd}
In the previous section, the quantum algorithms relied on the \textit{exact} simulation of continuous-time QHD dynamics.
However, in classical optimization, continuous-time dynamics (ODEs) are often used to model iterative algorithms; it is rare to directly simulate the ODE models to solve optimization problems. For example, gradient flow is considered a faithful continuous-time limit of gradient descent (GD), but they are seldom used as algorithms themselves. This is because GD is guaranteed to converge to a first-order stationary point, even with a relatively large step size. As a result, the trajectory of GD may deviate significantly from its continuous-time limit.

In this section, we discuss the time discretization of QHD dynamics and the corresponding convergence properties. Based on the standard first-order Trotter formula, we propose the discrete-time QHD algorithm. We study the behavior of discrete-time QHD for both smooth and non-smooth optimization problems and how step size is related to the global convergence.

\subsection{Quantum algorithms and complexity analysis}
A standard approach to simulating quantum evolution is to leverage the product formula (or operator splitting in numerical analysis). The essential idea of the (first-order) product formula is as follows. For a Hamiltonian operator $H = A + B$ and a small time $h$, we have the approximation: 
\begin{equation}\label{eqn:product-formula}
    e^{iHh} \approx e^{iAh}e^{iBh}.
\end{equation} 
In the previous section, we studied continuous-time QHD generated by the following Hamiltonian operators:
\begin{align}
    \hat{H}(t) = \frac{1}{\lambda(t)} \left(-\frac{1}{2}\Delta \right) + \lambda(t) f(x),
\end{align}
where the time-dependent function are chosen as: $\lambda(t) = e^{2\sqrt{\mu}t}$ for $\mu$-strongly convex $f$ (i.e., \textbf{QHD-SC}), $\lambda(t) = t^3$ for convex $f$ (i.e., \textbf{QHD-C}), and $\lambda(t) = \alpha t^{1/3}$ for non-convex $f$ (i.e., \textbf{QHD-NC}).
Observing that the Hamiltonian operator $\hat{H}(t)$ encompasses two components, namely, the kinetic operator $- \frac{1}{2}\Delta$ and the potential operator $f$. Since both operators can be efficiently simulated on a quantum computer, we can simulate the QHD dynamics via the product formula~\eqn{product-formula}, as summarized in~\algo{discrete-time-qhd}. 

\begin{algorithm}[H]
\begin{small}
\caption{Discrete-time QHD}
\label{algo:discrete-time-qhd}
\vspace{5pt}
\textbf{Classical inputs:} time-dependent parameter $\lambda(t)$ (\textbf{QHD-SC}: $\lambda(t) = e^{2\sqrt{\mu}t}$; \textbf{QHD-C}: $\lambda(t) = t^3$; \textbf{QHD-NC}: $\lambda(t) = \alpha t^{1/3}$), step size $h > 0$, number of iterations $K$\\
\textbf{Quantum inputs:} zeroth-order oracle $O_f$, an initial guess state $\ket{\Psi_0}$ \\
\textbf{Output:} an approximate solution to the minimization problem~\eqn{problem-form}
\vspace{7pt}
\hrule
\vspace{5pt}
\begin{algorithmic}

\For{$k\in\{1, ..., K\}$}
    \State Determine $t_k = k h$.
    \State Implement a quantum circuit $U_{k,1}$ that computes $e^{-i h \lambda(t_k)f}$.
    \State Implement a quantum circuit $U_{k,2}$ that computes $e^{i h \frac{1}{2\lambda(t_k)} \Delta}$.
    \State Apply the quantum circuits $U_{k,1}$ and $U_{k,2}$ to $\ket{\Psi_{k-1}}$ to obtain $\ket{\Psi_k} = U_{k,2}U_{k,1} \ket{\Psi_{k-1}}$.
\EndFor
\State Measure the quantum state $\ket{\Psi_k}$ using computational basis.
\end{algorithmic}
\vspace{5pt}
\end{small}
\end{algorithm}

\paragraph{Complexity analysis.}
In practice, the operators $f$ and $-\Delta / 2$ are unbounded operators (i.e., with infinite $L^2$-norm), and their corresponding quantum evolutions $e^{-i h \lambda(t_k)f}$ and $e^{i h \frac{1}{2\lambda(t_k)} \Delta}$ cannot be directly implemented by a digital quantum computer with a finite degree of freedom. 
To implement the quantum circuits $U_{k,1}$ and $U_{k,2}$, we need to perform spatial discretization in the real space $\R^d$ and reduce the simulation to a problem restricted to a fine mesh with $N$ points on each edge, a standard technique in numerical simulation of PDE. By the spatial discretization, the quantum states $\ket{\Psi_k}$ are now represented by a $N^d$-dimensional vector with a normalized $\ell^2$-norm. Thanks to the superposition nature of quantum computation, we can compute this exponential-size vector in a quantum computer with a computational cost logarithmically small in the size of the vector, as detailed in the following proposition.

\begin{proposition}\label{prop:discrete-time-qhd-complexity}
    Fix a spatial discretization number $N > 0$. Then,~\algo{discrete-time-qhd} can be implemented using $K$ quantum queries to the function value of $f$ and an additional $\mathcal{O}(dK\log(N))$ elementary quantum gates.
\end{proposition}
\begin{proof}
    The complexity analysis is essentially the same as in~\cite[Theorem 3]{leng2023qhd}.
    The unitary evolution $e^{i h \frac{1}{2\lambda(t_k)} \Delta}$ can be diagonalized by the Quantum Fourier Transform (QFT), and the eigenvalues can be explicitly computed with a circuit with $\mathcal{O}(d\log(N))$ gates. 
    The unitary $e^{-i h \lambda(t_k)f}$ can be implemented by just 1 use of the gradient oracle $O_{f}$. After $K$ iterations, the overall query complexity is $K$, with additional $\mathcal{O}(dK\log(N))$ elementary quantum gates.
\end{proof}

\subsection{Numerical study of convergence properties} \label{sec:convex-numerics}
In this section, we leverage numerical methods to analyze the convergence behavior of the discrete-time Quantum Hamiltonian Descent with constant step size.

\subsubsection{Warm-up: smooth and non-smooth optimization} \label{sec:disc-qhd}

In classical optimization theory, it is well-known that gradient descent converges with a fixed step size $h \le 1/ L$, where $L$ is the Lipschitz constant of $\nabla f$. 
In contrast, when applied to non-smooth optimization problems, a subgradient algorithm with a fixed step size $h$ will lead to a $\Omega(h)$ optimality gap, regardless of the number of iterations.

In the preceding section, we establish the asymptotic convergence rate of the QHD dynamics for both convex~(\thm{qhd-convex}) and strongly convex~(\thm{qhd-strongly-convex}) cases. These results indicate that the continuous-time dynamics will achieve an arbitrarily small optimality gap for a long enough evolution time $t$.
However, our numerical experiments show that the non-vanishing optimality gap re-emerges in the discrete-time QHD algorithm with constant step size.

\begin{figure}[htb!]
\centering
\begin{subfigure}[t]{0.48\linewidth}
\centering
\includegraphics[scale=0.48]{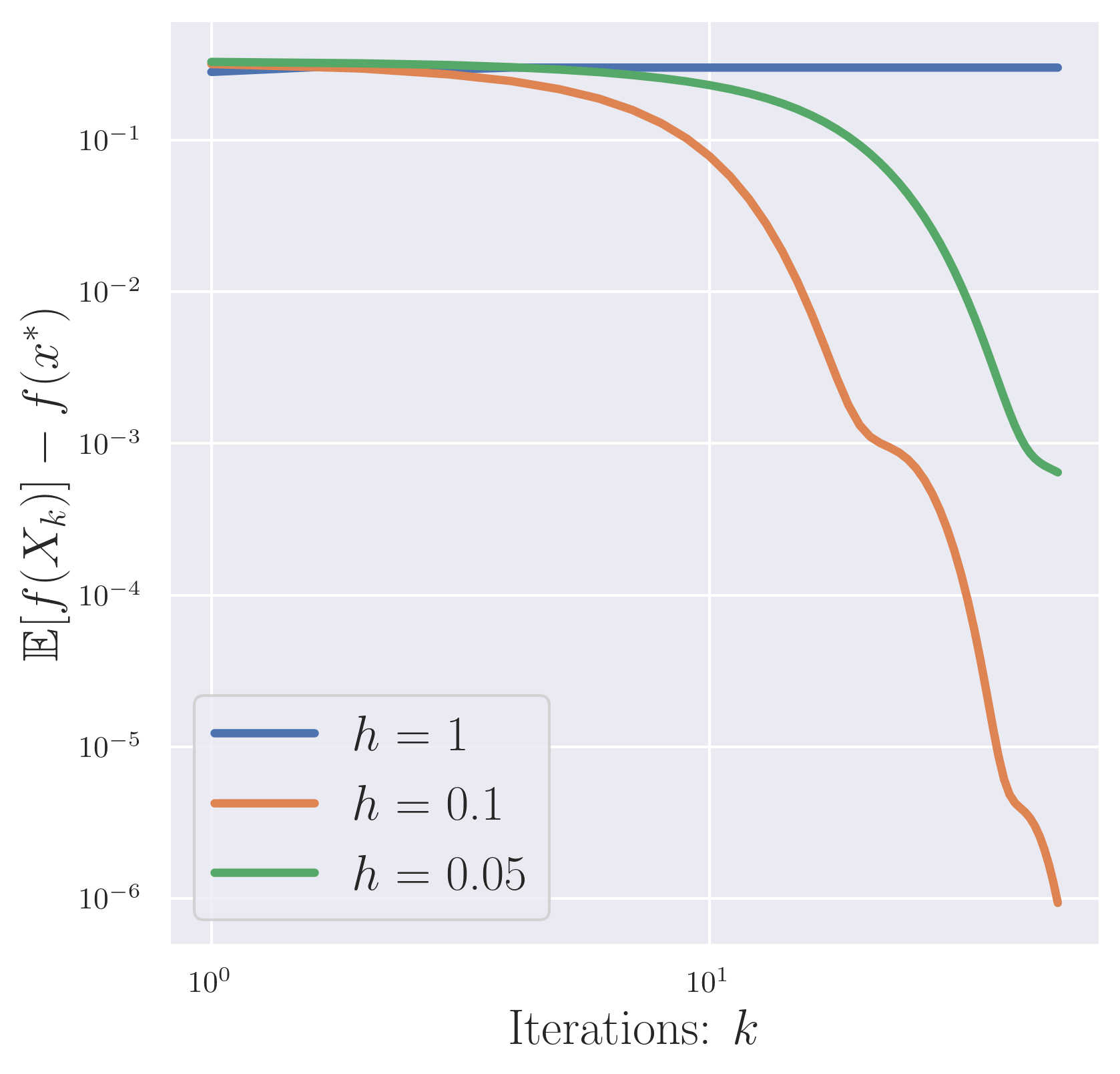}
\caption{$f(x) = x^2$.}
\label{fig:warm-up-a}
\end{subfigure}
\begin{subfigure}[t]{0.48\linewidth}
\centering
\includegraphics[scale=0.48]{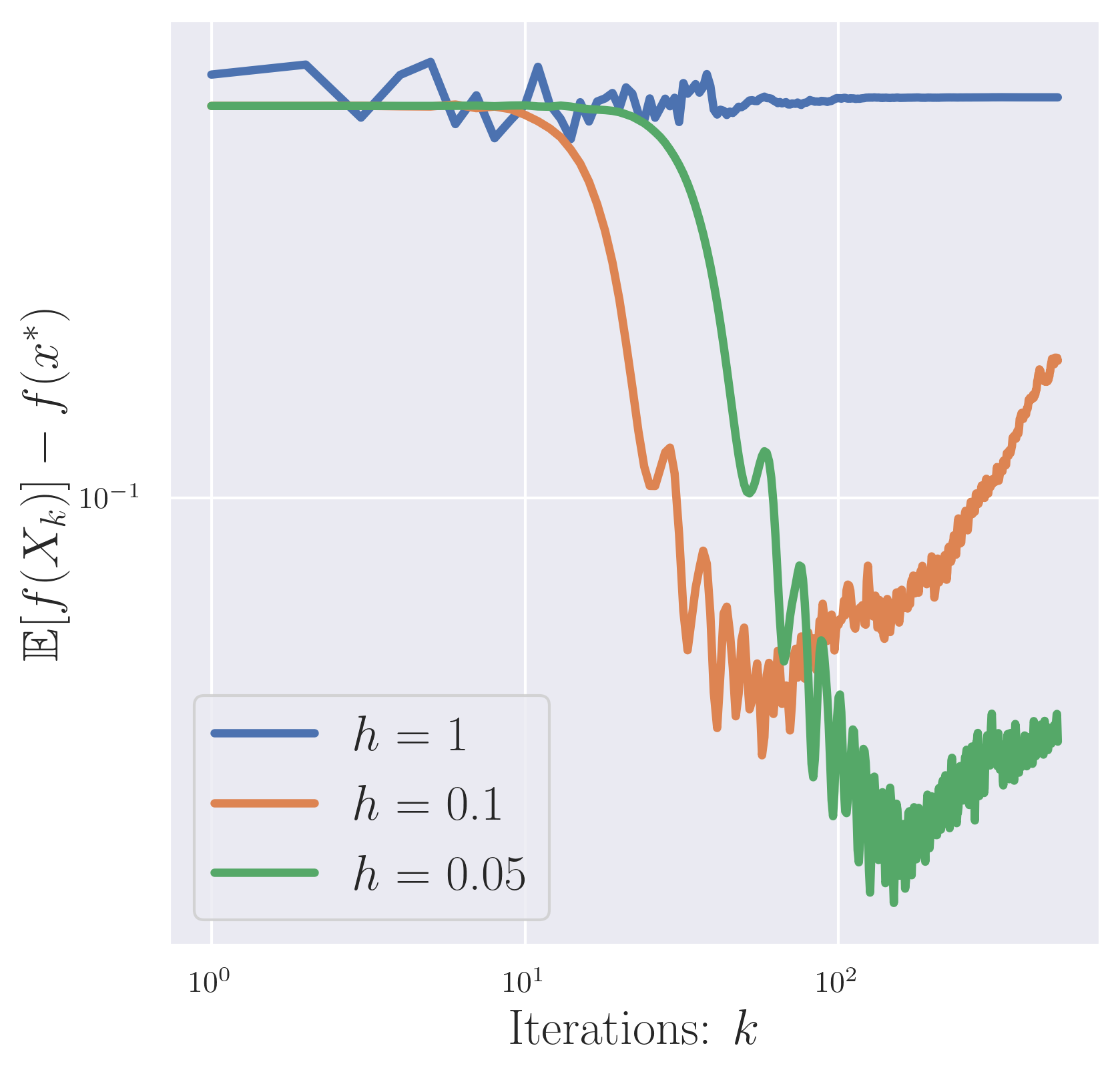}
\caption{$f(x) = |x|$.}
\label{fig:warm-up-b}
\end{subfigure}
\caption{Numerical demonstration of convergence behaviors of discrete-time QHD with fixed step size $h$ for smooth and non-smooth optimization.}
\label{fig:warmup}
\end{figure}

In~\fig{warmup}, we illustrate the change of the expected function value in discrete-time QHD with various step sizes $h$.
The left panel (\fig{warm-up-a}) depicts the performance of discrete-time QHD for a strongly convex smooth objective function $f(x) = x^2$. While the discrete-time QHD algorithm does not appear to converge with a large step size (e.g., $h = 1$), it exhibits a clear convergence pattern when the step size $h$ drops below a certain threshold, a phenomenon echoed with the standard gradient descent for smooth optimization. 

In the right panel (\fig{warm-up-b}), we illustrate the behavior of discrete-time QHD for a convex non-smooth objective function $f(x) = |x|$. We find that, even with a small step size (e.g., $h = 0.05$), the objective function value curve starts oscillating once it hits a sub-optimal value. Meanwhile, the terminal sub-optimal values of discrete-time QHD seem to depend on the step size $h$. With $h = 0.1$, the minimal value of the expected function value is achieved at around $0.03$; the minimal expected function value decreases to around $0.01$ with $h = 0.05$. The correlation between step sizes and optimality gaps in discrete-time QHD demonstrates a qualitative similarity to sub-gradient methods. 

In the subsequent section, we conduct more fine-grained numerical experiments to understand the convergence of discrete-time QHD with constant step size, focusing on the relations between the sub-optimality gap and the step size. A quantitative comparison with subgradient algorithms is provided in~\sec{comparison-subgrad}. 

\paragraph{Optimality gap.}
First, we formally define the optimality gap in the discrete-time QHD algorithm.
We denote $\Psi_k(x)$ as the quantum wave function after $k$ iterations in discrete-time QHD. By measuring the quantum register at this point, the outcome is a random vector $X_k \in \R^d$ distributed according to the probability density $|\Psi_k(x)|^2$.
The optimality gap achieved in the first $k$ iterations is characterized by
\begin{align}
   \mathcal{G}(k) \coloneqq f^{(k)}_{\text{best}} - f(x^*) =  \min_{1 \le j \le k}  \left(\mathbb{E}[f(X_k)] - f(x^*)\right).
\end{align}
Our numerical results depicted in~\fig{warmup} can be interpreted as follows. For smooth optimization problems, provided that the step size $h$ is below a threshold determined by the objective function $f$, the optimality gap $\mathcal{G}(k)$ converges to zero as the iteration number $k$ increases. 
On the other hand, the optimality gap of discrete-time QHD for non-smooth optimization will get stuck at a non-zero value no matter how many iterations are applied. Now, we investigate the quantitative relations between the step size $h$ and the optimality gap in discrete-time QHD for non-smooth optimization. 

\begin{figure}[htb!]
\centering
\begin{subfigure}[t]{0.48\linewidth}
\centering
\includegraphics[scale=0.48]{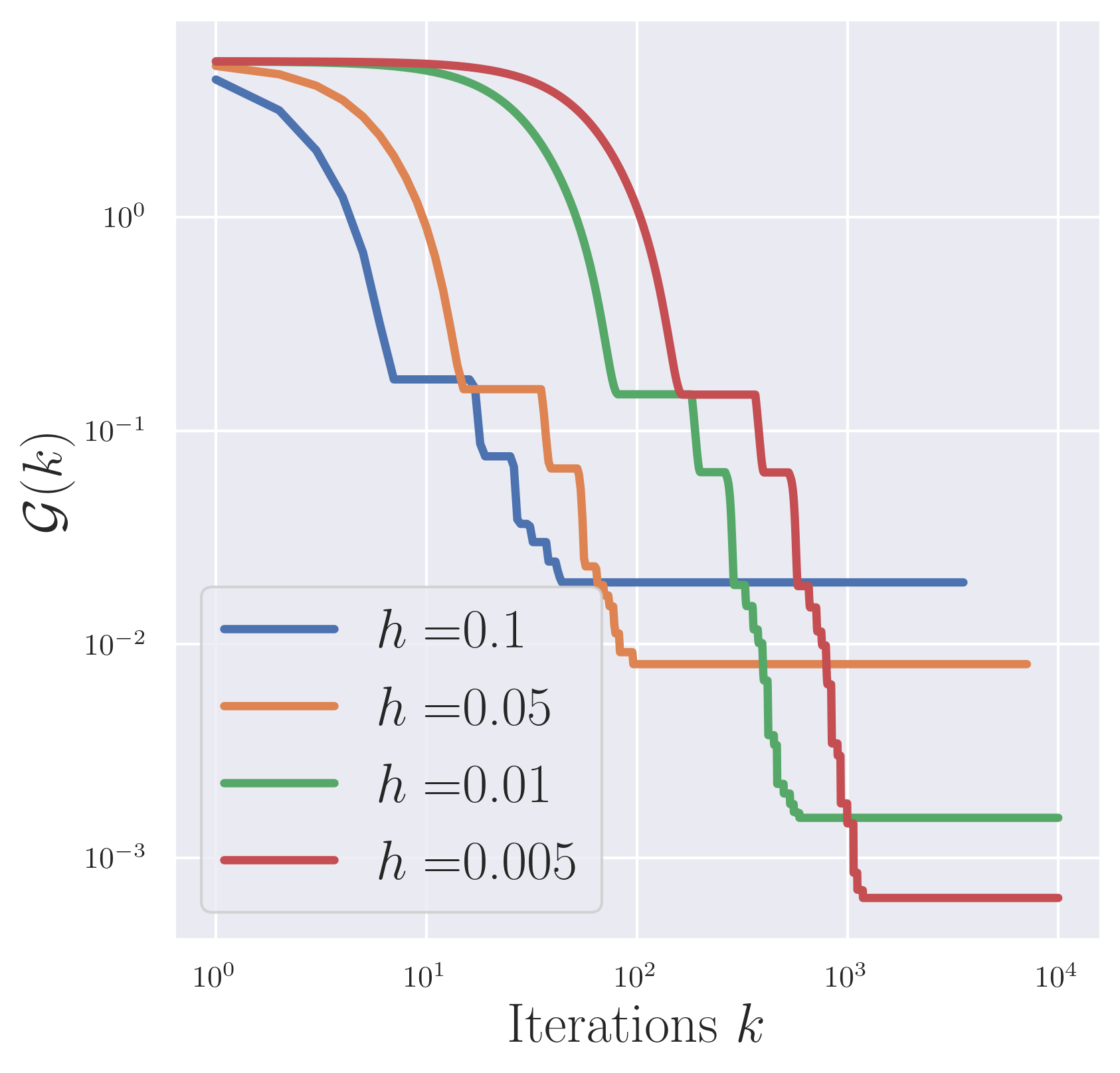}
\caption{Expected function values}
\label{fig:exp-abs-a}
\end{subfigure}
\begin{subfigure}[t]{0.48\linewidth}
\centering
\includegraphics[scale=0.48]{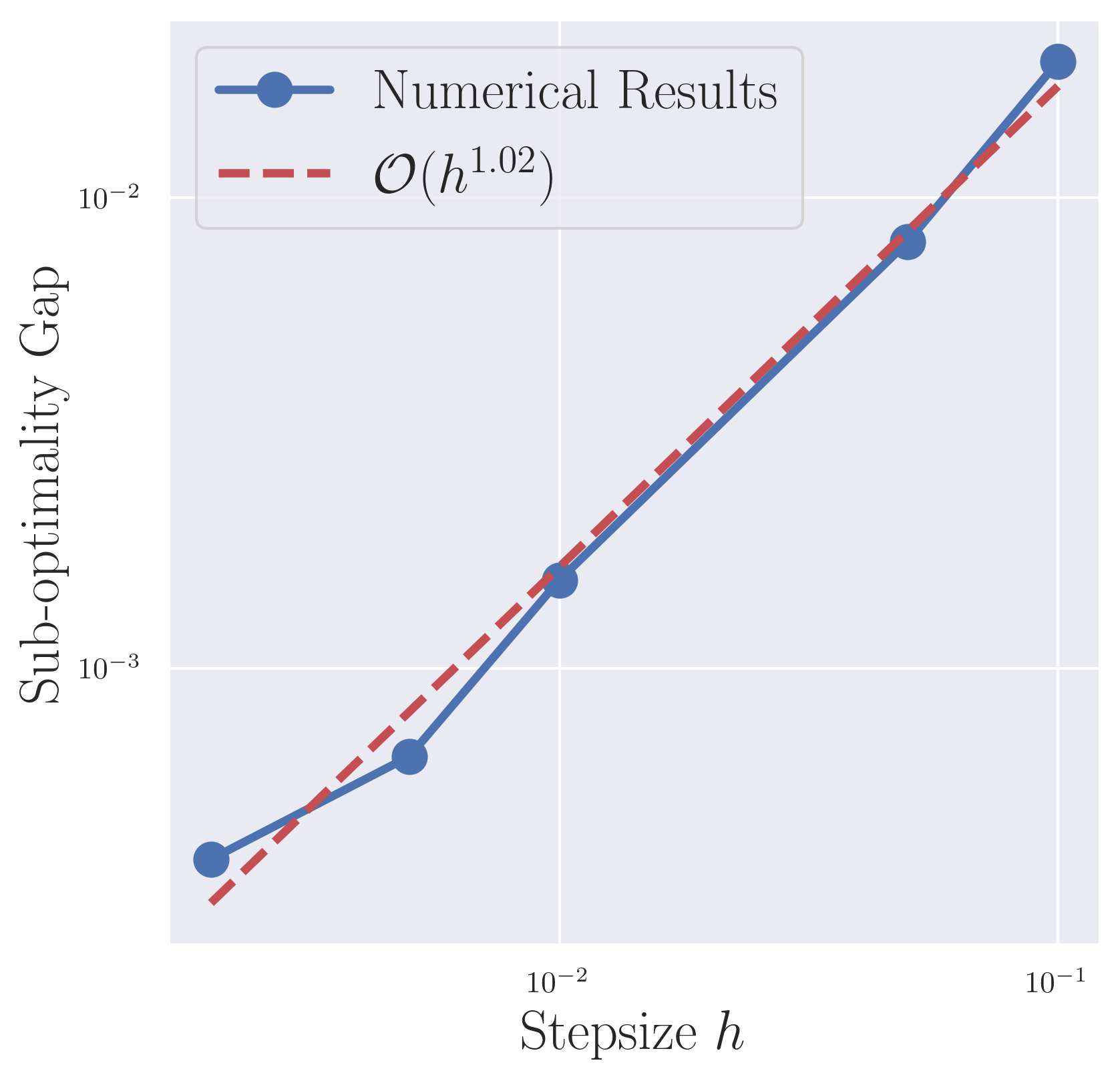}
\caption{Sub-optimality gaps}
\label{fig:exp-abs-b}
\end{subfigure}
\caption{Convergence behaviors of discrete-time QHD with various constant step sizes $h$ for $f(x) = e^{|x|} - 1$.}
\label{fig:exp-abs-convergence}
\end{figure}

\subsubsection{Case I: non-smooth strongly convex optimization}
In~\fig{exp-abs-convergence}, we illustrate the change of the optimality gap $\mathcal{G}(k)$ as a function of the iteration number $k$. We choose a non-smooth objective function $f(x) = e^{|x|}-1$, which is $1$-strongly-convex as we have $f''(x) \ge 1$. Similar to the previous case, the optimality gap first sees a sharp decrease with a linear convergence rate $\mathcal{O
}(e^{-k})$. As the iteration progresses, the optimality gap plateaus at a value linear in the step size $h$, as depicted in~\fig{exp-abs-b}. 
The two-stage evolution in the optimality gap in the strongly convex problem is consistent with our observations in the previous example, as summarized as follows:
\begin{align}\label{eqn:sc-optimality-gap}
    \mathcal{G}(k) \approx Ae^{-\sqrt{\mu}kh} + B h,
\end{align}
where the linear convergence rate $\mathcal{O}(e^{-\sqrt{\mu}kh})$ corresponds to the continuous-time case with a convergence rate $\mathcal{O}(e^{-\sqrt{\mu}t})$, and the terminal optimality gap scales linearly with the step size $h$.

\begin{figure}[htb!]
\centering
\begin{subfigure}[t]{0.48\linewidth}
\centering
\includegraphics[scale=0.48]{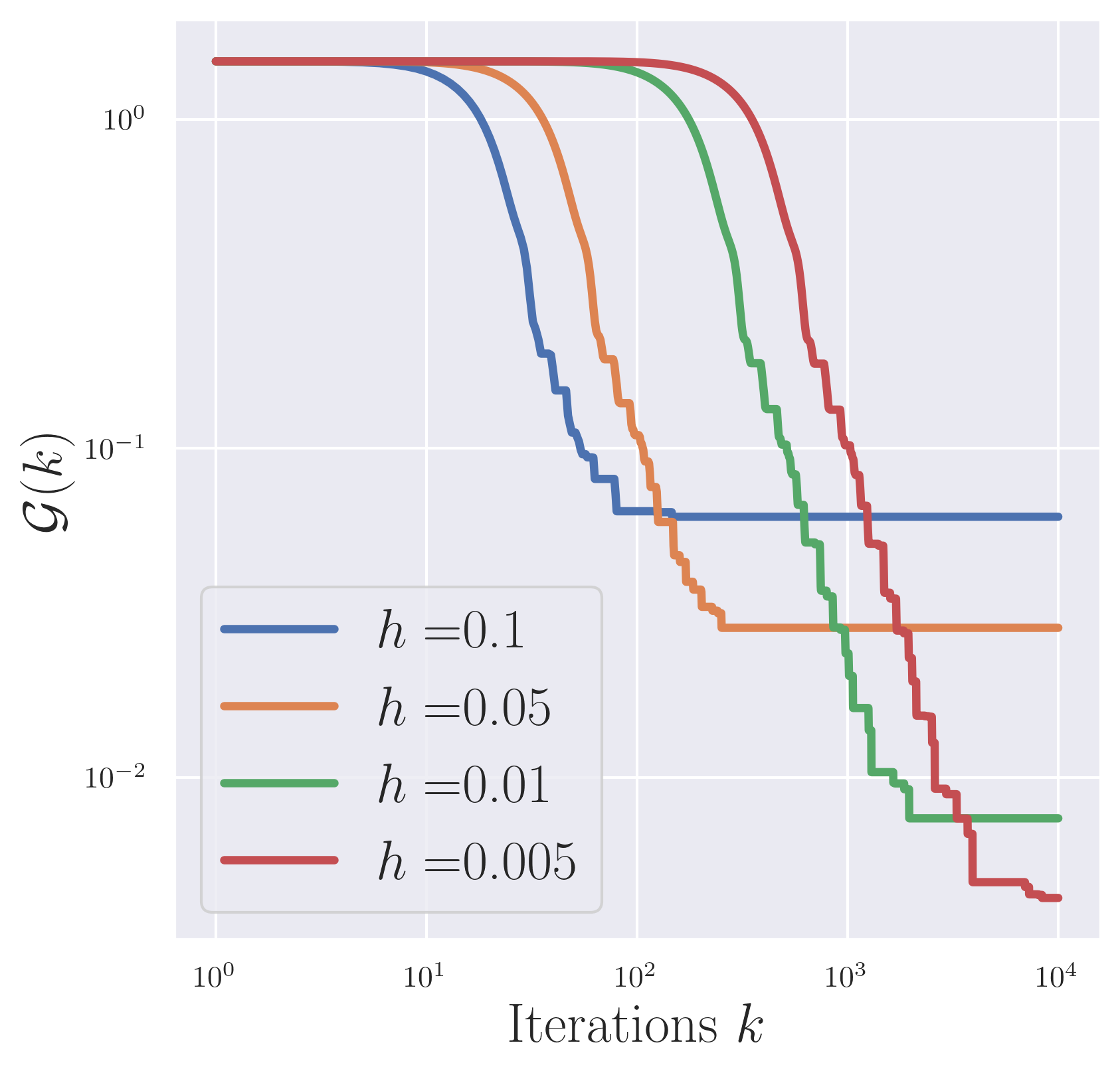}
\caption{Expected function values}
\label{fig:abs-val-a}
\end{subfigure}
\begin{subfigure}[t]{0.48\linewidth}
\centering
\includegraphics[scale=0.48]{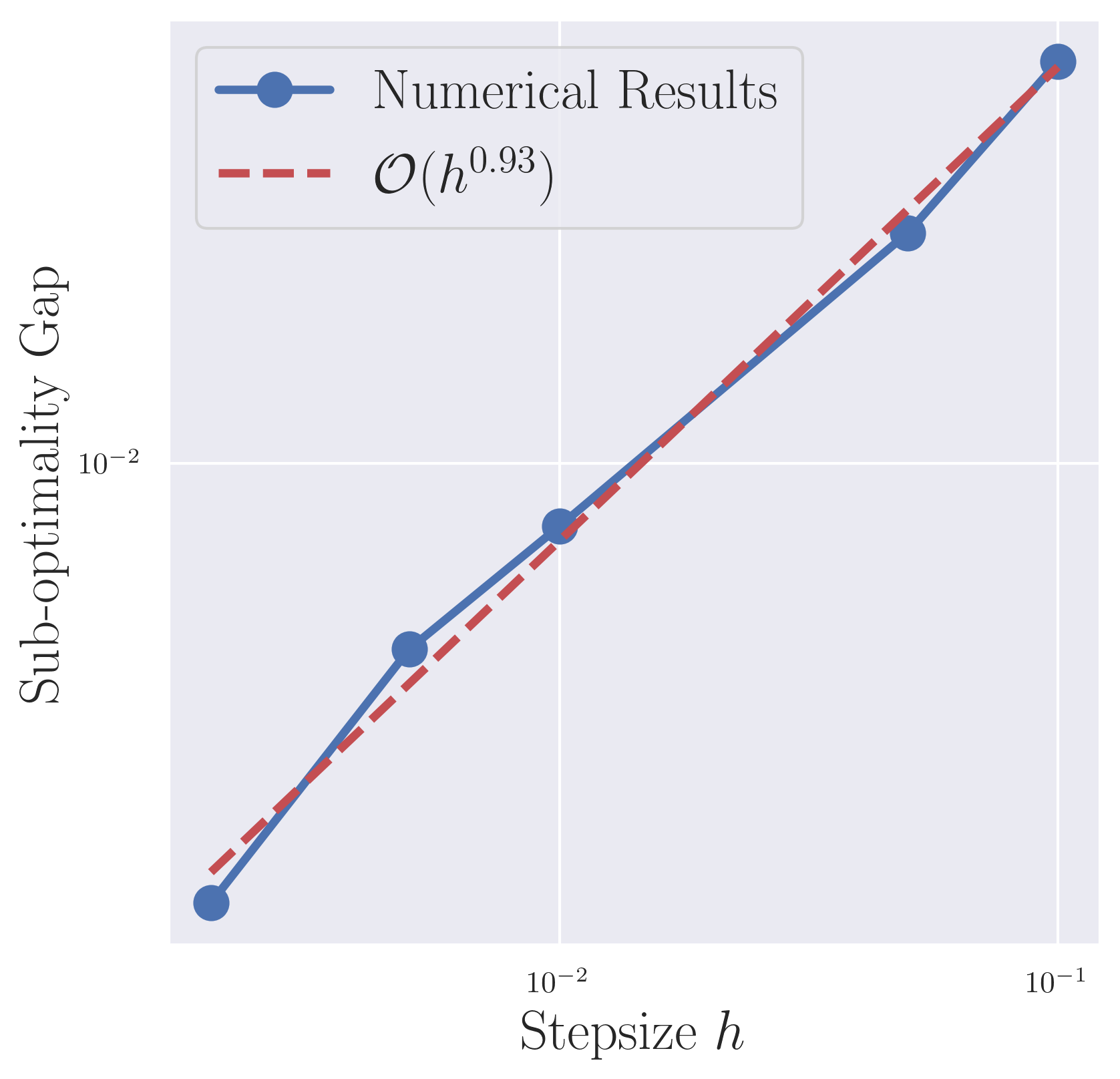}
\caption{Sub-optimality gaps}
\label{fig:abs-val-b}
\end{subfigure}
\caption{Convergence behaviors of discrete-time QHD with various constant step sizes $h$ for $f(x) = |x|$.}
\label{fig:abs-val-convergence}
\end{figure}

\subsubsection{Case II: non-smooth convex optimization}
In~\fig{abs-val-a}, we illustrate the change of the optimality gap $\mathcal{G}(k)$ as the iteration number $k$ grows. 
The numerical experiment is exemplified by a non-smooth objective function $f(x) = |x|$, where the optimality gap is computed with various step sizes $h$. 
We observe that the optimality gap first undergoes a convergence process with an approximate convergence rate $\mathcal{O}(k^{-2})$, which appears to be inherited from the continuous-time dynamics exhibiting a $O(t^{-2})$ convergence rate.
Then, as the iteration continues, the optimality gap plateaus.
\fig{abs-val-b} shows the terminal sub-optimality gap with different choices of $h$.
A linear fit reveals that the terminal optimality gap scales almost linearly with the step size $h$. 
In summary, the numerical observations suggest that following scaling in the optimality gap:
\begin{align}\label{eqn:convex-optimality-gap}
   \mathcal{G}(k) \approx \frac{A}{(kh)^2} + B h,
\end{align}
where $A$ and $B$ are two multiplicative constants that are independent of the step size $h$ and iteration number $k$.

\subsection{Discussion and limitation}\label{sec:comparison-subgrad}

\paragraph{Comparison with subgradient algorithms.}
Recall that the subgradient algorithm with constant step size $s$ achieves an optimality gap
\begin{equation}
    \frac{\|\xx_0 - \xx^*\|}{2sk} + \frac{L^2s}{2}
\end{equation}
after $k$ iterations. Therefore, to achieve an optimality gap $\epsilon$ in subgradient algorithm, we need to choose a step size $s \sim \epsilon$ and an iteration number $k = \mathcal{O}(1/\epsilon^2)$. 
Meanwhile, our numerical results~\eqn{convex-optimality-gap} suggest a faster convergence rate in terms of $k$, which implies that we may choose $h \sim \epsilon$ and 
\begin{equation}
    k = \mathcal{O}\left(\frac{1}{\epsilon^{3/2}}\right)
\end{equation}
in discrete-time QHD to achieve the same optimality gap $\epsilon$. This quantum convergence rate is a sub-quadratic speedup compared to the classical subgradient algorithm, and it also appears to beat the classical query lower bound $\Omega(1/\epsilon^2)$, as detailed in~\sec{subgrad-review}. This observation suggests a potential fundamental separation between quantum and classical optimization algorithms, and we plan to conduct a detailed investigation in future research.

For a $\mu$-strongly convex problem,~\thm{subg-cvg-s} indicates that the subgradient algorithm requires $k = \mathcal{O}(1/(\mu\epsilon))$ iterations to achieve an $\epsilon$ optimality gap. With our numerical extrapolation~\eqn{sc-optimality-gap}, we need to choose a step size $h \sim \epsilon$ and an iteration number 
\begin{equation}
    \mathcal{O}\left(\frac{1}{\sqrt{\mu}}\cdot \frac{\log(1/\epsilon)}{\epsilon}\right).
\end{equation}
Compared to the classical subgradient method, there is no speedup in terms of $1/\epsilon$ but a quadratic speedup in terms of $1/\mu$.

\paragraph{Consequences for quantum query complexity.}
For convex non-smooth optimization problems, the discrete-time QHD iteration number $k = \mathcal{O}(\epsilon^{-3/2})$
seem to contradict a no-go theorem in~\cite{garg2021no}, which states that no quantum algorithm can find an $\epsilon$-approximate minimum of a non-smooth convex function using $\widetilde{o}(\epsilon^{-2})$ queries to its value and subgradient.
Note that this $\widetilde{\Omega}(\epsilon^{-2})$ quantum query lower bound matches the bound for classical algorithms~\cite{nemirovski1983problem}, suggesting no quantum computational advantage could be achieved in this scenario.
However, we remark that the no-go theorem in~\cite{garg2021no} implicitly assumes that the worse-case non-smooth optimization problem instance has a large dimension 
$d = \widetilde{\Theta}(\epsilon^{-4})$.
Therefore, it is still possible that the discrete-time QHD can realize a query complexity better than $\mathcal{O}(\epsilon^{-2})$ for problem instances with dimension $d \ll \epsilon^{-4}$.
When it comes to the dependence on the dimension $d$, the best query complexity of classical algorithms is $\mathcal{O}(d \log \epsilon^{-1})$ given by the center of gravity method~\cite{bubeck2015convex}, which saturates the lower bound $\Omega(d \log \epsilon^{-1})$ shown in~\cite{nemirovski1983problem}.
It is still unclear if discrete-time QHD can achieve a lower query complexity for some parameter regimes of $d$ and $\epsilon$.

\section{Numerical Results for Non-Smooth Non-Convex Optimization} \label{sec:num}

In this section, we focus on the empirical performance of QHD for non-smooth non-convex optimization problems since we have already established the convergence of both continuous- and discrete-time QHD for non-smooth convex optimization in earlier sections.
We evaluate the performance of QHD and two variants of subgradient methods, namely, Subgrad (vanilla subgradient algorithm)\footnote{In our implementation of Subgrad, the algorithm returns the final iterate instead of the best among all iterates, requiring only \emph{one} function value query to $f(\cdot)$ throughout and allowing us to allocate other queries to $\partial_C f$ for a better performance with the same query number limit. This is a common practice in subgradient methods~\cite{zamani2023exact}.} and LFMSGD (learning rate-free momentum stochastic gradient descent)~\cite{hu2024learning}.
We test them on 12 non-smooth, non-convex, and box-constrained functions, including one 1-dimensional, nine 2-dimensional, and two 3-dimensional cases, as detailed in \secapp{A-func}.
For details on QHD, Subgrad, and LFMSGD, refer to \sec{qhd-review}, \sec{subgrad-review}, and \sec{A-classical-algo}, respectively.  
All source code is publicly available~\cite{git}. 

\subsection{Methodology} \label{sec:meth}

We use classical simulation of the discrete-time QHD introduced in \sec{disc-qhd} to evaluate the performance of QHD. 
For fairness among QHD, Subgrad, and LFMSGD, each algorithm is allowed a total of 10,000 queries to the function value (resp., subgradient) for QHD (resp. Subgrad and LFMSGD) in a single run.
Subgrad and LFMSGD will start with 10,000 initial points (in 10,000 independent runs) chosen from the domain uniformly at random.
The experiment is conducted on a consumer laptop (Intel Core i7-8750H CPU 2.20GHz).

\paragraph{Parameter Setup.}
For QHD, we introduce a parameter $L$ to rescale a test function into hypercube $[-L,L]^d$.
Specifically, a test function \( g: \bigtimes_{1 \leq i \leq d}[l_i, u_i] \to \mathbb{R} \) is transformed into the function \( f : [-L,L]^d \to \mathbb{R} \) as follows:  
\begin{equation} \label{eqn:func-scaling}
    f(x_1,\dots,x_d) \coloneqq g ( x_1',\dots,x_d' ), \quad x_i' \coloneqq l_i + \frac{u_i-l_i}{2L}(x_i + L).
\end{equation} 
We parameterize QHD with \( L \) and set \( \lambda(t) = t^3 \) in~\eqn{cont-time-qhd-general}, with total evolution time \( T=10 \), time step size \( h = 1/1000 \), and spatial discretization number \( N = 512 \).\footnote{Adjusting \( L \) is equivalent to modifying \( \lambda(t) \) and \( T \) in QHD; refer to \secapp{A-qhd-L}.}  
We choose the same \( \lambda(t) \) as in \textbf{QHD-C}~\eqn{qhd-c} rather than \textbf{QHD-NC}~\eqn{qhd-nc} because it yields significantly stronger empirical performance.  
This is analogous to the case of stochastic gradient descent (SGD), where learning rate schedules used in practice bear little resemblance to those recommended by theory~\cite{defazio2023optimal}.
For Subgrad parameterized by $\eta$, the learning rate follows \( s_j = \eta/\sqrt{j} \) in light of \thm{subg-cvg}.  
LFMSGD, parameterized by $\sigma$, is configured with the default setting \( \beta = 0.9 \) as recommended in~\cite{hu2024learning}, and the Gaussian noise added to each evaluation of the subgradient follows $\mathcal{N}(\boldsymbol{0},\sigma^2)$.

\paragraph{Performance Metric.} 
All three algorithms---QHD, Subgrad, and LFMSGD---exhibit randomness: QHD and LFMSGD have internal sources of randomness, while LFMSGD and Subgrad depend on the random initial point.
To evaluate their performance, we use the \emph{best-of-$k$ optimality gap} as a metric.\footnote{Here, $k$ denotes the number of runs and should not be confused with its usage as the iteration number in previous sections.} This metric is inspired by a common practice of running a randomized algorithm multiple times and selecting the best result. 
Formally, for an integer $k$, the best-of-$k$ optimality gap is defined as:
\begin{equation}
    \mathbb{E}\left[\min_{1 \leq i \leq k} f(\tilde{\xx}_i) - f_{\rm min}\right],
\end{equation}
where $f(\cdot)$ is the test function, $f_{\rm min} \coloneqq \min_{\xx} f(\xx)$ is the global minimum, $\tilde{\xx}_i$ is the solution returned by the algorithm in the $i$th run, and the expectation is taken with respect to the algorithm's internal randomness and/or random initialization. 
The metric reduces to the expected optimality gap when $k=1$. 
We test all three algorithms on 12 test functions with $k \in \{1, 3, 10, 30, 100\}$.
The best-of-$k$ optimality gap is estimated using the Monte Carlo method for Subgrad and LFMSGD. 
In contrast, for QHD, it is computed exactly since the full distribution of the final output of the algorithm is known in our classical simulation.

\paragraph{Parameter Optimization.} 
QHD, Subgrad, and LFMSGD are parameterized and their performance depend heavily on the choice of corresponding parameters $L$, $\eta$, and $\sigma$, respectively.
Therefore, we employ Bayesian optimization~\cite{snoek2012practical} to select the optimal parameter that minimizes the best-of-$k$ optimality gap. 
Bayesian optimization is particularly suited for hyperparameter tuning in settings where function evaluations are costly. 
For each combination of $k \in \{1, 3, 10, 30, 100\}$ and test function, we limit the number of evaluations of the best-of-$k$ optimality gap to 100 when optimizing parameters.

\subsection{Benchmarking QHD against classical algorithms}

\begin{figure}[htbp]
        \centering
        \includegraphics[height=0.9\textheight]{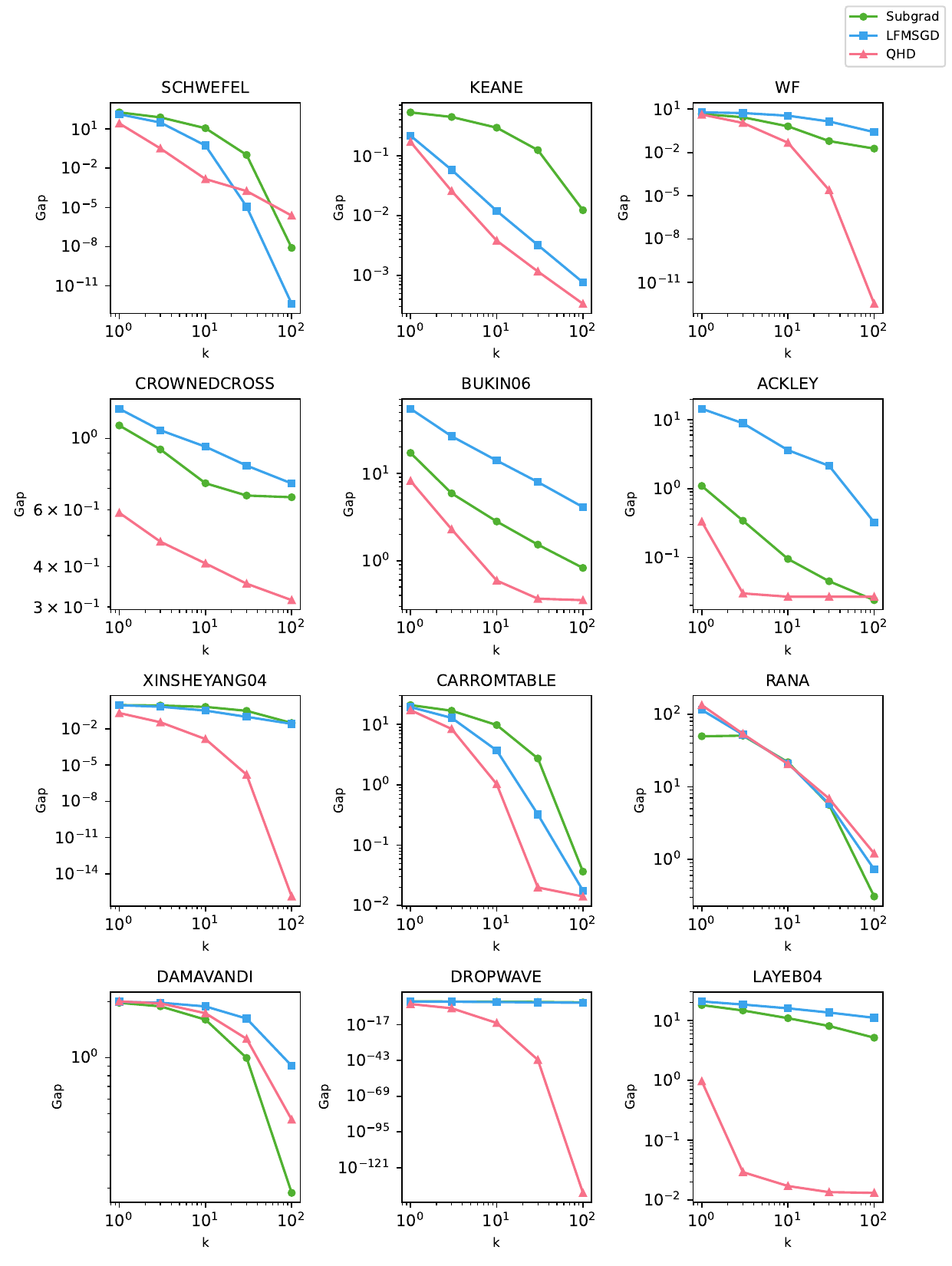} 
        \caption{Best-of-$k$ optimality gap for 12 test functions (cf. \tab{best-of-k-table}). QHD outperforms Subgrad and LFMSGD for all choices of $k$ on every test functions except \texttt{SCHWEFEL}, \texttt{ACKLEY}, \texttt{RANA} and \texttt{DAMAVANDI}.}
        \label{fig:best-of-k}
\end{figure}

\begin{table}[htbp]
\centering
\resizebox{\textwidth}{!}{%
\begin{tabular}{|ccccc|l|ccccc|}
\hline
\textbf{Test Function}                                                                      & $\boldsymbol{k}$   & \textbf{QHD Gap}                        & \textbf{LFMSGD Gap}                      & \textbf{Subgrad Gap}                    &  & \textbf{Test Function}                                                                      & $\boldsymbol{k}$   & \textbf{QHD Gap}                          & \textbf{LFMSGD Gap} & \textbf{Subgrad Gap}                    \\ \hline
                                                                                       & \textbf{1}   & {\color[HTML]{3166FF} \textbf{2.72e+1}} & \textbf{1.37e+2}                         & \textbf{1.90e+2}                        &  &                                                                                        & \textbf{1}   & {\color[HTML]{3166FF} \textbf{1.70e-1}}   & \textbf{2.17e-1}    & \textbf{5.31e-1}                        \\
                                                                                       & \textbf{3}   & {\color[HTML]{3166FF} \textbf{3.29e-1}} & \textbf{3.18e+1}                         & \textbf{7.91e+1}                        &  &                                                                                        & \textbf{3}   & {\color[HTML]{3166FF} \textbf{2.58e-2}}   & \textbf{5.81e-2}    & \textbf{4.45e-1}                        \\
                                                                                       & \textbf{10}  & {\color[HTML]{3166FF} \textbf{1.49e-3}} & \textbf{5.51e-1}                         & \textbf{1.17e+1}                        &  &                                                                                        & \textbf{10}  & {\color[HTML]{3166FF} \textbf{3.81e-3}}   & \textbf{1.20e-2}    & \textbf{2.95e-1}                        \\
                                                                                       & \textbf{30}  & \textbf{1.79e-4}                        & {\color[HTML]{3166FF} \textbf{1.15e-5}}  & \textbf{1.05e-1}                        &  &                                                                                        & \textbf{30}  & {\color[HTML]{3166FF} \textbf{1.17e-3}}   & \textbf{3.22e-3}    & \textbf{1.24e-1}                        \\
\multirow{-5}{*}{\textbf{\begin{tabular}[c]{@{}c@{}}SCHWEFEL\\ (1D)\end{tabular}}}     & \textbf{100} & \textbf{2.37e-6}                        & {\color[HTML]{3166FF} \textbf{4.20e-13}} & \textbf{7.98e-9}                        &  & \multirow{-5}{*}{\textbf{\begin{tabular}[c]{@{}c@{}}KEANE\\ (2D)\end{tabular}}}        & \textbf{100} & {\color[HTML]{3166FF} \textbf{3.37e-4}}   & \textbf{7.67e-4}    & \textbf{1.23e-2}                        \\ \hline
                                                                                       & \textbf{1}   & {\color[HTML]{3166FF} \textbf{3.30e-1}} & \textbf{1.45e+1}                         & \textbf{1.09e+0}                        &  &                                                                                        & \textbf{1}   & \textbf{1.33e+2}                          & \textbf{1.14e+2}    & {\color[HTML]{3166FF} \textbf{4.96e+1}} \\
                                                                                       & \textbf{3}   & {\color[HTML]{3166FF} \textbf{2.98e-2}} & \textbf{8.93e+0}                         & \textbf{3.41e-1}                        &  &                                                                                        & \textbf{3}   & \textbf{5.35e+1}                          & \textbf{5.21e+1}    & {\color[HTML]{3166FF} \textbf{5.05e+1}} \\
                                                                                       & \textbf{10}  & {\color[HTML]{3166FF} \textbf{2.67e-2}} & \textbf{3.62e+0}                         & \textbf{9.45e-2}                        &  &                                                                                        & \textbf{10}  & {\color[HTML]{3166FF} \textbf{2.05e+1}}   & \textbf{2.10e+1}    & \textbf{2.17e+1}                        \\
                                                                                       & \textbf{30}  & {\color[HTML]{3166FF} \textbf{2.67e-2}} & \textbf{2.14e+0}                         & \textbf{4.48e-2}                        &  &                                                                                        & \textbf{30}  & \textbf{6.86e+0}                          & \textbf{5.78e+0}    & {\color[HTML]{3166FF} \textbf{5.62e+0}} \\
\multirow{-5}{*}{\textbf{\begin{tabular}[c]{@{}c@{}}ACKLEY\\ (2D)\end{tabular}}}       & \textbf{100} & \textbf{2.67e-2}                        & \textbf{3.22e-1}                         & {\color[HTML]{3166FF} \textbf{2.38e-2}} &  & \multirow{-5}{*}{\textbf{\begin{tabular}[c]{@{}c@{}}RANA\\ (2D)\end{tabular}}}         & \textbf{100} & \textbf{1.21e+0}                          & \textbf{7.28e-1}    & {\color[HTML]{3166FF} \textbf{3.08e-1}} \\ \hline
                                                                                       & \textbf{1}   & {\color[HTML]{3166FF} \textbf{8.19e+0}} & \textbf{5.48e+1}                         & \textbf{1.71e+1}                        &  &                                                                                        & \textbf{1}   & {\color[HTML]{3166FF} \textbf{4.25e+0}}   & \textbf{6.01e+0}    & \textbf{4.66e+0}                        \\
                                                                                       & \textbf{3}   & {\color[HTML]{3166FF} \textbf{2.29e+0}} & \textbf{2.65e+1}                         & \textbf{5.90e+0}                        &  &                                                                                        & \textbf{3}   & {\color[HTML]{3166FF} \textbf{1.12e+0}}   & \textbf{5.41e+0}    & \textbf{2.78e+0}                        \\
                                                                                       & \textbf{10}  & {\color[HTML]{3166FF} \textbf{5.92e-1}} & \textbf{1.40e+1}                         & \textbf{2.81e+0}                        &  &                                                                                        & \textbf{10}  & {\color[HTML]{3166FF} \textbf{4.73e-2}}   & \textbf{3.54e+0}    & \textbf{6.48e-1}                        \\
                                                                                       & \textbf{30}  & {\color[HTML]{3166FF} \textbf{3.67e-1}} & \textbf{7.95e+0}                         & \textbf{1.53e+0}                        &  &                                                                                        & \textbf{30}  & {\color[HTML]{3166FF} \textbf{2.48e-5}}   & \textbf{1.41e+0}    & \textbf{6.17e-2}                        \\
\multirow{-5}{*}{\textbf{\begin{tabular}[c]{@{}c@{}}BUKIN06\\ (2D)\end{tabular}}}      & \textbf{100} & {\color[HTML]{3166FF} \textbf{3.54e-1}} & \textbf{4.10e+0}                         & \textbf{8.26e-1}                        &  & \multirow{-5}{*}{\textbf{\begin{tabular}[c]{@{}c@{}}WF\\ (2D)\end{tabular}}}           & \textbf{100} & {\color[HTML]{3166FF} \textbf{3.28e-13}}  & \textbf{2.57e-1}    & \textbf{1.86e-2}                        \\ \hline
                                                                                       & \textbf{1}   & {\color[HTML]{3166FF} \textbf{1.73e+1}} & \textbf{1.95e+1}                         & \textbf{2.10e+1}                        &  &                                                                                        & \textbf{1}   & {\color[HTML]{3166FF} \textbf{2.11e-1}}   & \textbf{9.34e-1}    & \textbf{9.63e-1}                        \\
                                                                                       & \textbf{3}   & {\color[HTML]{3166FF} \textbf{8.45e+0}} & \textbf{1.29e+1}                         & \textbf{1.69e+1}                        &  &                                                                                        & \textbf{3}   & {\color[HTML]{3166FF} \textbf{3.63e-2}}   & \textbf{7.06e-1}    & \textbf{8.92e-1}                        \\
                                                                                       & \textbf{10}  & {\color[HTML]{3166FF} \textbf{1.03e+0}} & \textbf{3.72e+0}                         & \textbf{9.81e+0}                        &  &                                                                                        & \textbf{10}  & {\color[HTML]{3166FF} \textbf{1.47e-3}}   & \textbf{3.38e-1}    & \textbf{6.84e-1}                        \\
                                                                                       & \textbf{30}  & {\color[HTML]{3166FF} \textbf{2.00e-2}} & \textbf{3.24e-1}                         & \textbf{2.74e+0}                        &  &                                                                                        & \textbf{30}  & {\color[HTML]{3166FF} \textbf{1.63e-6}}   & \textbf{1.02e-1}    & \textbf{3.23e-1}                        \\
\multirow{-5}{*}{\textbf{\begin{tabular}[c]{@{}c@{}}CARROMTABLE\\ (2D)\end{tabular}}}  & \textbf{100} & {\color[HTML]{3166FF} \textbf{1.41e-2}} & \textbf{1.77e-2}                         & \textbf{3.65e-2}                        &  & \multirow{-5}{*}{\textbf{\begin{tabular}[c]{@{}c@{}}XINSHEYANG04\\ (2D)\end{tabular}}} & \textbf{100} & {\color[HTML]{3166FF} \textbf{1.27e-16}}  & \textbf{2.67e-2}    & \textbf{3.14e-2}                        \\ \hline
                                                                                       & \textbf{1}   & {\color[HTML]{3166FF} \textbf{5.88e-1}} & \textbf{1.24e+0}                         & \textbf{1.10e+0}                        &  &                                                                                        & \textbf{1}   & {\color[HTML]{3166FF} \textbf{1.06e-2}}   & \textbf{6.90e-1}    & \textbf{8.04e-1}                        \\
                                                                                       & \textbf{3}   & {\color[HTML]{3166FF} \textbf{4.78e-1}} & \textbf{1.06e+0}                         & \textbf{9.25e-1}                        &  &                                                                                        & \textbf{3}   & {\color[HTML]{3166FF} \textbf{8.93e-6}}   & \textbf{5.02e-1}    & \textbf{6.93e-1}                        \\
                                                                                       & \textbf{10}  & {\color[HTML]{3166FF} \textbf{4.09e-1}} & \textbf{9.43e-1}                         & \textbf{7.26e-1}                        &  &                                                                                        & \textbf{10}  & {\color[HTML]{3166FF} \textbf{1.81e-16}}  & \textbf{2.99e-1}    & \textbf{5.22e-1}                        \\
                                                                                       & \textbf{30}  & {\color[HTML]{3166FF} \textbf{3.54e-1}} & \textbf{8.24e-1}                         & \textbf{6.65e-1}                        &  &                                                                                        & \textbf{30}  & {\color[HTML]{3166FF} \textbf{2.64e-43}}  & \textbf{1.69e-1}    & \textbf{3.57e-1}                        \\
\multirow{-5}{*}{\textbf{\begin{tabular}[c]{@{}c@{}}CROWNEDCROSS\\ (2D)\end{tabular}}} & \textbf{100} & {\color[HTML]{3166FF} \textbf{3.15e-1}} & \textbf{7.25e-1}                         & \textbf{6.57e-1}                        &  & \multirow{-5}{*}{\textbf{\begin{tabular}[c]{@{}c@{}}DROPWAVE\\ (3D)\end{tabular}}}     & \textbf{100} & {\color[HTML]{3166FF} \textbf{4.75e-140}} & \textbf{8.94e-2}    & \textbf{1.71e-1}                        \\ \hline
                                                                                       & \textbf{1}   & \textbf{2.00e+0}                        & \textbf{1.99e+0}                         & {\color[HTML]{3166FF} \textbf{1.97e+0}} &  &                                                                                        & \textbf{1}   & {\color[HTML]{3166FF} \textbf{9.65e-1}}   & \textbf{2.08e+1}    & \textbf{1.81e+1}                        \\
                                                                                       & \textbf{3}   & \textbf{1.95e+0}                        & \textbf{1.97e+0}                         & {\color[HTML]{3166FF} \textbf{1.88e+0}} &  &                                                                                        & \textbf{3}   & {\color[HTML]{3166FF} \textbf{2.92e-2}}   & \textbf{1.85e+1}    & \textbf{1.47e+1}                        \\
                                                                                       & \textbf{10}  & \textbf{1.73e+0}                        & \textbf{1.88e+0}                         & {\color[HTML]{3166FF} \textbf{1.60e+0}} &  &                                                                                        & \textbf{10}  & {\color[HTML]{3166FF} \textbf{1.71e-2}}   & \textbf{1.60e+1}    & \textbf{1.09e+1}                        \\
                                                                                       & \textbf{30}  & \textbf{1.26e+0}                        & \textbf{1.62e+0}                         & {\color[HTML]{3166FF} \textbf{9.96e-1}} &  &                                                                                        & \textbf{30}  & {\color[HTML]{3166FF} \textbf{1.35e-2}}   & \textbf{1.36e+1}    & \textbf{8.09e+0}                        \\
\multirow{-5}{*}{\textbf{\begin{tabular}[c]{@{}c@{}}DAMAVANDI\\ (2D)\end{tabular}}}    & \textbf{100} & \textbf{4.65e-1}                        & \textbf{9.04e-1}                         & {\color[HTML]{3166FF} \textbf{1.88e-1}} &  & \multirow{-5}{*}{\textbf{\begin{tabular}[c]{@{}c@{}}LAYEB04\\ (3D)\end{tabular}}}      & \textbf{100} & {\color[HTML]{3166FF} \textbf{1.31e-2}}   & \textbf{1.11e+1}    & \textbf{5.16e+0}                        \\ \hline
\end{tabular}%
}
\caption{Best-of-$k$ optimality gap summary for QHD, LFMSGD, and Subgrad on 12 test functions (cf. \fig{best-of-k}), including one 1-dimensional, nine 2-dimensional, and two 3-dimensional cases. The smallest gaps among three algorithms are marked blue.}
\label{tab:best-of-k-table}
\end{table}



As shown in \fig{best-of-k}, QHD consistently outperforms Subgrad and LFMSGD across 8 out of 12 test functions, achieving significantly smaller optimality gaps for all tested values of \( k \). Notably, for \texttt{WF}, \texttt{XINSHEYANG04} and \texttt{DROPWAVE}, QHD demonstrates an overwhelming advantage, reducing the gap by multiple orders of magnitude compared to the classical methods for large $k$. For \texttt{KEANE}, \texttt{CROWNEDCROSS}, \texttt{BUKIN06}, \texttt{CARROMTABLE} and \texttt{LAYEBO4}, QHD also achieves superior performance, maintaining a decisive lead across increasing values of \( k \). These results highlight the effectiveness of QHD in finding high-quality solutions where classical algorithms struggle to close the optimality gap, emphasizing its robustness across a diverse set of non-convex test functions.
The numerical data corresponding to \fig{best-of-k} is presented in \tab{best-of-k-table}.

For \texttt{SCHWEFEL} and \texttt{ACKLEY}, QHD initially maintains an advantage over the classical methods but fails to sustain it as \( k \) increases. Specifically, for \texttt{SCHWEFEL}, QHD outperforms both Subgrad and LFMSGD when \( k \leq 10 \) but loses its lead for \( k \geq 30 \). This can be attributed to the structure of \texttt{SCHWEFEL}, where the global minimum has a large basin of attraction, making it easier for any local search algorithm to succeed with randomized initialization and enough repetitions. In contrast, for \texttt{ACKLEY}, QHD outperforms the classical methods up to \( k = 30 \) but stagnates at \( k = 100 \). This is primarily due to the spatial discretization in our QHD simulation, which limits its ability to ``see'' solutions closer to the global minimum; see \secapp{A-qhd-bn} for a detailed discussion.

For the remaining two test functions, \texttt{RANA} and \texttt{DAMAVANDI}, QHD yields an optimality gap comparable with classical methods, and no significant advantage is observed for QHD. This suggests that these functions may not exhibit structures where QHD can succesfully utilize.
Nonetheless, QHD’s competitive performance on them, combined with its superiority on the majority of test functions, reinforces its effectiveness as a powerful optimization algorithm.

\subsection{QHD dynamics: a comparative case study}

\begin{figure}[htbp]
    \centering
    \includegraphics[trim=0pt 35pt 0pt 60pt, clip, width=0.8\linewidth]{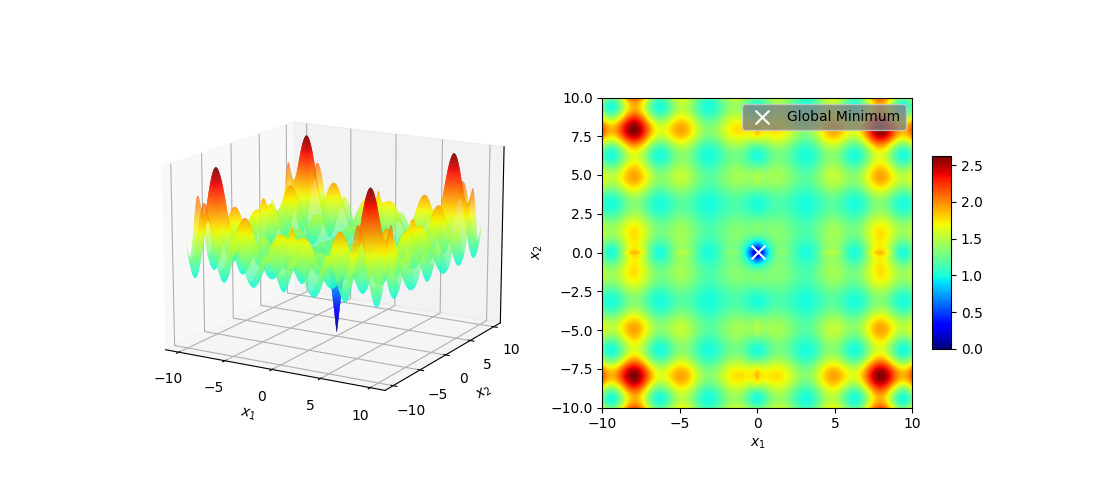}
    \caption{Landscape of test function \texttt{XINSHEYANG04}}
    \label{fig:xsy-landscape}
\end{figure}

\begin{figure}[ht!]
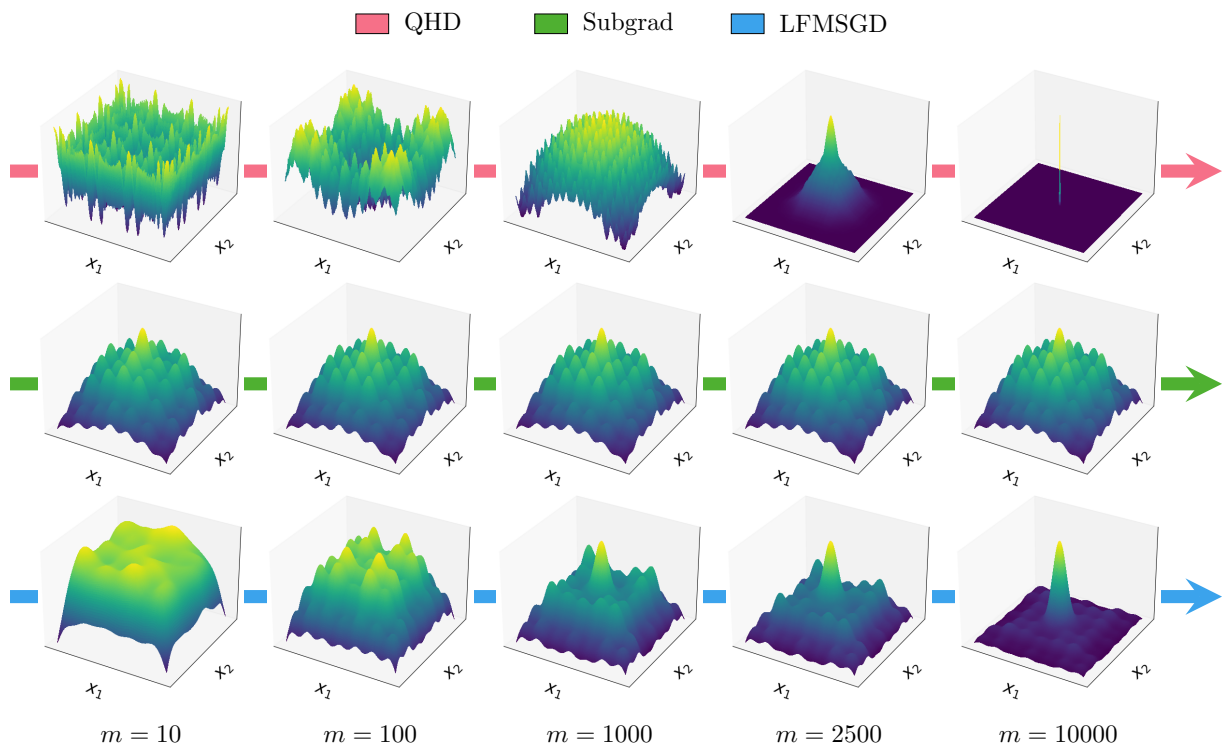

    \centering
    \resizebox{\textwidth}{!}{
    \begin{tikzpicture}
        \definecolor{qhdcolor}{RGB}{246, 112, 136}   
        \definecolor{subgradcolor}{RGB}{79, 176, 49} 
        \definecolor{lfmsgdcolor}{RGB}{59, 163, 236}  

        \draw[qhdcolor, line width=2mm, -stealth] (-2, 0.5) -- (16.5, 0.5);

        \draw[subgradcolor, line width=2mm, -stealth] (-2, -2.75) -- (16.5, -2.75);

        \draw[lfmsgdcolor, line width=2mm, -stealth] (-2, -6) -- (16.5, -6);
        
        \foreach \i/\m in {0/10, 1/100, 2/1000, 3/2500, 4/10000} {
            \node at (\i * 3.5, 0.5) {
                \begin{subfigure}[t]{0.19\textwidth}
                    \includegraphics[trim=70pt 50pt 55pt 75pt, clip, width=\linewidth]{figures/xinsheyang04_dist/QHD_\m.png}
                \end{subfigure}
            };
        }

        \foreach \i/\m in {0/10, 1/100, 2/1000, 3/2500, 4/10000} {
            \node at (\i * 3.5, -2.75) {
                \begin{subfigure}[t]{0.19\textwidth}
                    \includegraphics[trim=70pt 50pt 55pt 75pt, clip, width=\linewidth]{figures/xinsheyang04_dist/SUBGRAD_\m.png}
                \end{subfigure}
            };
        }

        \foreach \i/\m in {0/10, 1/100, 2/1000, 3/2500, 4/10000} {
            \node at (\i * 3.5, -6) {
                \begin{subfigure}[t]{0.19\textwidth}
                    \includegraphics[trim=70pt 50pt 55pt 75pt, clip, width=\linewidth]{figures/xinsheyang04_dist/LFMSGD_\m.png}
                \end{subfigure}
            };
        }

        \node at (0, -8.1) {\textbf{$m=10$}};
        \node at (3.5, -8.1) {\textbf{$m=100$}};
        \node at (7, -8.1) {\textbf{$m=1000$}};
        \node at (10.5, -8.1) {\textbf{$m=2500$}};
        \node at (14, -8.1) {\textbf{$m=10000$}};

        \node[anchor=center] at (7.4, 2.75) { 
            \begin{tikzpicture}
                \draw[fill=qhdcolor] (0.27, 0) rectangle (0.77, 0.3);
                \node[anchor=west] at (0.87, 0.15) {QHD};
        
                \draw[fill=subgradcolor] (3, 0) rectangle (3.5, 0.3);
                \node[anchor=west] at (3.6, 0.15) {Subgrad};
        
                \draw[fill=lfmsgdcolor] (6, 0) rectangle (6.5, 0.3);
                \node[anchor=west] at (6.6, 0.15) {LFMSGD};
            \end{tikzpicture}
        };
    \end{tikzpicture}
    }

    \caption{Distributions at intermediate iterations for QHD (1st row), Subgrad (2nd row), and LFMSGD (3rd row) on \texttt{XINSHEYANG04} where $m$ denotes the iteration number. Note that in QHD, $m = t / h = 1000 t$ due to time discretization (see \sec{disc-qhd}), where $t$ denotes the evolution time.}
    \label{fig:i-dist}
\end{figure}

To illustrate how QHD outperforms Subgrad and LFMSGD in most cases, we need to show the intermediate iterations for these three algorithms. We use test function \texttt{XINSHEYANG04} as an example. The landscape of the 2-dimensional multimodal function \texttt{XINSHEYANG04} has its global minimum located at the center as shown in \fig{xsy-landscape}. We also illustrate the intermediate iteration process of three algorithms on test function \texttt{XINSHEYANG04} in \fig{i-dist}. We summarize our observations below.

QHD demonstrates a strong ability to explore the landscape and escape local minima. Initially, the distribution appears highly fluctuating, reflecting a \emph{kinetic phase}, during which the algorithm averages the initial wave function over the whole search space to reduce the risk of poor initialization~\cite{leng2023qhd}. As the evolution time $t$ (or iteration number $m$) increases, QHD dynamically adapts, gradually concentrating on more promising regions while retaining some level of exploration (\emph{global search phase}). By the end of the evolution, the distribution becomes more localized (\emph{descent phase}), indicating convergence toward the global minimum while avoiding premature trapping in suboptimal regions. Eventually, the distribution narrows into a sharp peak at the global minimum.

LFMSGD also exhibits the ability to escape local minima, albeit being less efficient. In the early iterations, the distributions are relatively broad due to the random initialization. However, unlike QHD, the refinement process is less efficient, as evidenced by the gradual and somewhat diffusive concentration of probability mass over subsequent iterations. As the evolution progresses, the distribution becomes more peaked, though some residual spread persists. Eventually, LFMSGD converges, but the slower contraction of the distribution suggests a less directed search compared to QHD, highlighting its inefficiency despite its stochastic capability to escape local minima.

Subgrad, in contrast, exhibits a starkly different behavior. Once the algorithm enters a basin of attraction, determined by the random initialization, it remains trapped in the corresponding local minimum encountered early in the process, and this entrapment persists across all iterations. This aligns with known limitations of subgradient methods, which lack intrinsic mechanisms for escaping local optima, making them ineffective at navigating non-convex landscapes.

In summary, QHD demonstrates the most effective escape dynamics and convergence efficiency, leveraging its quantum-inspired evolution to navigate complex landscapes. LFMSGD benefits from stochasticity to avoid local traps but at the cost of slower refinement. Subgradient descent, being fully deterministic, remains consistently trapped, reinforcing its limitations in highly non-convex settings.

\section{Conclusion}
In this work, we study how quantum computers can be leveraged to address non-smooth optimization problems. Specifically, we investigate the theoretical and empirical properties of the Quantum Hamiltonian Descent (QHD) algorithm. We propose three variants of the original continuous-time QHD algorithm and prove their global convergence. Inspired by the product formula in quantum simulation, we also introduce discrete-time QHD, a fully digitized implementation that achieves convergence similar to its continuous-time counterpart without requiring precise simulation of the original dynamics. Through various numerical experiments, we demonstrate that QHD outperforms classical subgradient algorithms in both convergence rate and solution quality. 

Several questions related to QHD still remain open. First, the convergence rate of discrete-time QHD is primarily established via numerical methods, and a rigorous proof appears both challenging and beyond the scope of this work. 
Second, the characterization of the runtime of (continuous-time) QHD for non-smooth non-convex problems relies on an abstract \textit{spectral gap}, and it would be appealing to relate this quantity with more specific structures of the objective function, such as weak convexity or Hessian bounds.
Third, due to the limited computational power of classical computers, we can currently only simulate QHD on up to two-dimensional optimization problems. With improved numerical simulation techniques, we anticipate that the numerical study of QHD could achieve better scalability and shed light on high-dimensional problems in practice.

\bibliographystyle{plainnat}
\bibliography{ref}

\appendix

\section{Subgradient methods}\label{sec:subgrad-review}
A function $f\colon \R^d \to \R$ is \textit{convex} if for all $0 \le t \le 1$ and any $\xx_1, \xx_2 \in \R^d$, we have
\begin{equation}
    f(t\xx_1 + (1-t)\xx_2) \le tf(\xx_2) + (1-t)f(\xx_2).
\end{equation}
The gradient of a non-smooth function $f$ is not well-defined everywhere. Instead, we consider a weaker notion of subgradient.
Let $f\colon \R^d \to \R$. We say a vector $\gg\in \R^d$ is a \textit{subgradient} of $f$ at $\xx$ if,
\begin{equation}
    f(\yy) \ge f(\xx) + \gg\trans (\yy-\xx), \quad \forall \yy \in \R^d.
\end{equation}
The subgradient of a function $f$ at a fixed point is not unique. The \textit{subdifferential} of $f$ at $\xx \in \R^d$ is defined as the set of all subgradients of $f$ at $\xx$, denoted by
\begin{equation}
    \partial f(\xx) \coloneqq \{\gg\in\R^d\colon f(\yy) \ge f(\xx) + \gg\trans (\yy-\xx), \quad \forall \yy \in \R^d\}.
\end{equation}
We say $f$ is \textit{subdifferentiable} at $\xx$ if $\partial f(\xx) \neq \varnothing$.
When $f$ is convex and finite over $\R^d$, the subdifferential $\partial f(\xx)$ is a convex, bounded, and closed set for every $\xx$. The directional derivative $f'(\xx, \vv)$ at any $\xx$ along the direction $\vv$ can be represented by
\begin{equation}
    f'(\xx,\vv) = \lim_{t\to 0^+} \frac{f(\xx + t \vv) - f(\xx)}{t} = \max_{\gg \in \partial f(\xx)} \gg\trans \vv.
\end{equation}
In particular, when $f$ is differentiable at $\xx$, $\partial f(\xx) = \{\nabla f(\xx)\}$ and we recover the following formula from standard calculus,
\begin{equation}
    f'(\xx,\vv) = \nabla f(\xx)\trans \vv.
\end{equation}
A point $\xx^*$ is a local minimizer of a function $f$ (not necessarily convex) if and only if $f$ is subdifferentiable at $\xx^*$ and $0 \in \partial f(\xx^*)$.

We say a function $f\colon \R^d \to \R$ is $L$-\textit{Lipschitz} if for any $\xx, \yy \in \R^d$, 
\begin{equation}
    |f(\xx) - f(\yy)| \le L \|\xx - \yy\|.
\end{equation}
If $f$ is differentiable, the gradient of $f$ must be bounded by the Lipschitz constant, i.e., $\|\nabla f\|_\infty \le L$. 
A function $f$ is \textit{locally Lipschitz} if for any $\xx_0 \in \R^d$, there exists a local Lipschitz constant $\ell_{x_0} > 0$ and a radius $\delta > 0$ such that
\begin{equation}
    |f(\xx) - f(\xx_0)| \le \ell_{x_0} \|\xx - \xx_0\|,\quad \forall \|\xx - \xx_0\| \le \delta.
\end{equation}
It is well-known that a locally Lipschitz function is differentiable almost everywhere.

Clarke's gradient ~\cite{clarke1990optimization} is a generalization of the notion of subdifferential that enables the analysis of
non-convex and non-smooth functions through convex analysis.
Let $f\colon \R^d \to \R$ be a locally Lipschitz function and $\Omega_f$ be the set such that $f$ is not differentiable. By the local Lipschitz property, the set $\Omega_f$ is of measure $0$. For a point $\xx \in \R^d$, we define the \textit{generalized gradient} (or \textit{Clarke's gradient}) of the function $f$ as follows:
\begin{equation}
    \partial_C f(\xx) \coloneqq \conv\left(\{\gg\colon \gg \in \R^d, \nabla f(\xx_k) \to \gg, \xx_k \to \xx, \xx_k \notin \Omega_f\}\right),
\end{equation}
where $\conv(A)$ is the \textit{convex hull} of a set $A$. 
A point $\xx^* \in \R^d$ is called a \textit{Clarke stationary point} if $0 \in \partial_C f(\xx^*)$. Note that Clarke stationarity is a necessary condition for a local optimum (including both minimum and maximum), but the inverse is not true.

\vspace{4mm}
The \textit{subgradient algorithm} generalizes the gradient descent algorithm by replacing the negative gradient update at each step with a negative subgradient update:
\begin{equation}
    \xx_{k+1} = \xx_k - s_k g_k, \ \forall k = 0, 1, \cdots
\end{equation}
where $\xx_0$ is an initial guess, $g_k \in \partial f(\xx_k)$ is \textit{any} subgradient of $f$ at $\xx_k$, and $s_k > 0$ is a step size.
In contrast to gradient descent, the subgradient algorithm is in general not a descent method (i.e., \{$f(\xx_k)\}_k$ is not monotonically decreasing). Therefore, it is natural to track the current best solution $\xx^{(k)}_{\mathrm{best}}$ and best function values $f^{(k)}_{\mathrm{best}}$:
\begin{equation}
    \xx^{(k)}_{\mathrm{best}} \coloneqq \argmin_{1 \le k \le K} f(\xx_k),\quad f^{(k)}_{\mathrm{best}} \coloneqq \min_{1 \le k \le K} f(\xx_k).
\end{equation}
With a constant step size $s > 0$, it has been proven that the subgradient algorithm has the following asymptotic convergence result for a convex function $f$ with a Lipschitz constant $L$: 
\begin{equation}
    f^{(k)}_{\mathrm{best}} - f(\xx^*) \le \frac{\|\xx_0 - \xx^*\|^2}{2sk} + \frac{L^2s}{2},
\end{equation}
where $\xx^*$ is the minimizer of $f$.
In this result, we observe a non-vanishing optimality gap $L^2 s/2 = \Omega(s)$ independent of the iteration number $k$. This behavior of the subgradient algorithm is radically different from gradient descent, as the optimality gap of the latter algorithm converges to $0$ at a rate $\mathcal{O}(k^{-1})$.

In practice, it is more favorable to consider a diminishing sequence of step sizes $\{s_k\}$, which will yield a vanishing optimality gap. In the following theorem, we consider a sequence of step sizes decrease as $k^{-1/2}$, and it turns out that the optimality gap also converges at the same speed.

\begin{theorem}[{\cite{bubeck2015convex}}] \label{thm:subg-cvg}
    Suppose that $f$ is a convex function such that $\norm{g} \leq L$ for any $\xx$ and $g \in \partial f(\xx)$.
    Let $\xx^\ast$ be the minimizer of $f$, $R \coloneqq \|\xx_0 - \xx^\ast\|$, and choose a step size $s_k = \frac{R}{L \sqrt{k}}$. Then, we have
    \begin{equation}
        f^{(k)}_{\mathrm{best}} - f(\xx^*) \leq
        RL \cdot \frac{\ln{k}}{\sqrt{k}}.
    \end{equation}
\end{theorem}

For strongly convex optimization problems, the convergence rate can be improved to $\mathcal{O}(k^{-1})$ provided that we choose a sequence of step size that decay at the same rate:

\begin{theorem}[{\cite{beck2017first}}] \label{thm:subg-cvg-s}
    Suppose that $f$ is a $\mu$-strongly convex function such that $\norm{g} \leq L$ for any $\xx$ and $g \in \partial f(\xx)$.\footnote{This requires that $f$ is restricted on a bounded region in $\mathbb{R}^d$. The subgradient algorithm referenced here is thus the projected subgradient algorithm~\cite{beck2017first}, a natural extension that enforces constraints by projecting iterates onto the convex feasible region.}
    Let $\xx^\ast$ be the minimizer of $f$ and choose a step size $s_k = \frac{2}{\mu (k+1)}$. 
    Then, we have
    \begin{equation}
        f^{(k)}_{\mathrm{best}} - f(\xx^*) \leq \frac{2L^2}{\mu} \cdot \frac{1}{k+1}.
    \end{equation}
\end{theorem}

The $k$-dependence in \thm{subg-cvg} and \thm{subg-cvg-s} is tight up to a factor of $\ln k$ for first-order methods: As established in~\cite{nemirovski1983problem}, for any $k$, there exist convex functions (or strongly convex functions, respectively) for which any first-order method fails to achieve an optimality gap $o(1/\sqrt{k})$ (or $o(1/k)$, respectively) using $k$ queries.\footnote{We write $f = o(g)$ if and only if $f/g \to 0$.}

\section{Analysis of Schr\"odinger equations}
\subsection{Spectral theory of Schr\"odinger operators} \label{sec:sdg-spec}

The standard theory of \sdg\ operators guarantees they have nice spectral property under minimal assumptions:

\begin{lemma}[{\cite[Theorem XIII.47 and XIII.67]{reed1972methods}}] \label{lem:sdg-spec}
    Let $V\colon \mathbb{R}^d \to \mathbb{R}$ be continuous\footnote{The continuity assumption can be weakened to $V \in L^2_{\rm loc}(\mathbb{R}^d)$; See~\cite{reed1972methods}.} and assume $\lim_{\xx \to \infty}V(\xx) = +\infty$.
    Then, $H = -\Delta + V$ has purely discrete spectrum with a nondegenerate strictly positive ground state.
    As a consequence, $H$ has a positive \emph{spectral gap} which is defined as the difference between the smallest and the second smallest (counting multiplicities) eigenvalues of $H$.
\end{lemma}
\begin{remark} \label{rem:sdg-spec}
    \lem{sdg-spec} can be extended to the case where $V$ is defined on $\Omega \coloneqq [-1,1]^d$: in such a case we no longer need any assumption for $\lim_{\xx \to +\infty} V(\xx)$. To understand this, it is important to note that the spectrum will be unchanged if we extend $V$ to $\widehat{V}$, where
    \begin{equation}
        \widehat{V}(\xx) = \left\{
            \begin{array}{ll}
                V(\xx), & x \in \Omega; \\
                +\infty, & x \notin \Omega.
            \end{array}
        \right.
    \end{equation}
    Now that $\widehat{V}$ is defined on $\mathbb{R}^d$, the only obstacle for applying \lem{sdg-spec} is that $\widehat{V}$ is not continuous.
    This can be tackled by a standard method that constructs a continuous function that is sufficiently close to $\widehat{V}$.
\end{remark}

\subsection{Regularity theory of Schr\"odinger equations}\label{append:regularity}
In this section, we show that QHD preserves the smoothness of the quantum state in the evolution. Specifically, we consider the following time-dependent Schr\"odinger equation,
\begin{align}\label{eqn:schrodinger-basic}
    i \partial_t \Psi(t,\xx) = \left(-\frac{1}{2\lambda(t)}\Delta + \lambda(t) f(\xx)\right)\Psi(t,\xx),
\end{align}
where $\xx \in \Omega = [0,1)^d$ with periodic boundary condition, $f(\xx) \in L^\infty(\Omega)$, and $\lambda(t)\colon [0, T]\to \R^+$ is a continuous function.

\begin{remark}
    While we consider unconstrained optimization problems in this paper, the minimum of the objective function is often distributed in a compact subset in $\R^d$, which means it is only necessary for an optimization algorithm to search within a finite bounding box. 
    In QHD, this means we can restrict the quantum evolution in a box since the quantum wave function is essentially vanishing outside of this bounding box. Therefore, it is reasonable to model QHD by a Schr\"odinger equation with periodic boundary conditions (up to proper rescaling of the space).
\end{remark}

\begin{definition}
    For any $s >0$, we define the $H^s$-norm of a function $u \in L^2(\Omega)$ by
    \begin{align}
        \|u\|_{H^s} = \left(\sum_{\kk \in \Z^d} (1 + 4\pi^2|\kk|^2)^s |\hat{u}(\kk)|^2\d \kk\right)^{1/2},
    \end{align}
    where $\hat{u}(\kk)$ is the Fourier transform of the function $u(\xx)$,
    \begin{align}
        \hat{u}(\kk) = \int_\Omega u(\xx) e^{-2i\pi \xx \cdot \kk}~\d \xx.
    \end{align}
\end{definition}

\begin{lemma}[Sobolev estimates]\label{lem:sobolev-estimates}
    Let $\Psi(t,\xx) \colon [0,T] \to L^2(\Omega)$ be the solution to~\eqn{schrodinger-basic}. Then, for any $s > 0$ and $t \in [0, T]$, there is a constant $C_{T} > 0$ such that
    \begin{align}
        \|\Psi(t)\|_{H^s} \le C_T \|\Psi_0\|_{H^s}.
    \end{align}
\end{lemma}
\begin{proof}
    Direct calculation shows that 
    \begin{align}\label{eqn:sobolev-1}
        \frac{\d}{\d t}\|\Psi(t)\|^2_{H^s} = 2 \Re{\langle \partial_t \Psi(t), (1-\Delta)^s \Psi(t)\rangle},
    \end{align}
    where the operator $(1-\Delta)^s$ is defined by 
    \begin{align}
        (1-\Delta)^s u \coloneqq \mathcal{F}^{-1}\left[\sum_{\kk \in \Z^d}(1 + 4\pi^2|\kk|^2)^s \mathcal{F}(u)(\kk)\right],
    \end{align}
    where $\mathcal{F}(\cdot)$ and $\mathcal{F}^{-1}(\cdot)$ represents Fourier transform and inverse Fourier transform, respectively.
    We plug~\eqn{schrodinger-basic} into~\eqn{sobolev-1}, and it follows that
    \begin{align}
        \frac{\d}{\d t}\|\Psi(t)\|^2_{H^s} = 2 \Re\left(\left\langle \left(-i\frac{1}{2\lambda(t)} \Delta  - i\lambda(t) f\right)\Psi(t), (1-\Delta)^s \Psi(t)\right\rangle\right)
    \end{align}
    Note that the self-adjoint operators $\Delta$ commutes with $(1-\Delta)^s$ for any $s > 0$, therefore $\langle \Delta \Psi(t), (1-\Delta)^s \Psi(t)\rangle$ is a real number for any $t$, and  
    \begin{align}
        \frac{\d}{\d t}\|\Psi(t)\|^2_{H^s} = 2\lambda(t) \Re\left(\left\langle -if \Psi(t), (1-\Delta)^s \Psi(t)\right\rangle\right) \le 2\lambda(t) \|f\|_{\infty} \|\Psi(t)\|^2_{H^s},
    \end{align}
    where the last step follows from H\"older's inequality. 
    Next, we invoke the Gr\"onwall's inequality and obtain
    \begin{align}
        \|\Psi(t)\|^2_{H^s} \le \exp\left(\int^t_0 \lambda(s) \|f\|_\infty\d s\right) \|\Psi(0)\|^2_{H^s}. 
    \end{align}
    Thus, let $C_T \coloneqq \exp\left(\int^T_0 \lambda(s) \|f\|_\infty\d s\right)$, we complete the proof of the lemma. 
\end{proof}

\begin{lemma}[Sobolev embedding theorem]\label{lem:sobolev-embedding}
    If $s - k > d/2$, then we have $H^s(\R^d) \hookrightarrow C^k(\R^d)$.
\end{lemma}
\begin{proof}
    This is a standard result in functional analysis, see~\cite{brezis2011functional} for details.
\end{proof}
\begin{remark}
    Note that the $\hookrightarrow$ symbol means the embedding of the Sobolev space in a larger space. Here, embedding not only means the Sobolev space $H^s$ is completely included in $C^k$ as a set but also the $H^s$-norm naturally extends to $C^k$.
\end{remark}

\begin{theorem}\label{thm:regularity}
    Consider the time-dependent Schr\"odinger equation~\eqn{schrodinger-basic} subject to a smooth initial condition $\Psi_0(\xx) \in C^\infty(\Omega)$.
    Then, the PDE~\eqn{schrodinger-basic} admits a unique solution such that $\Psi(t,\xx) \in C^\infty(\R^d)$ for all finite $t \ge 0$.
\end{theorem}
\begin{proof}
    Note that the initial state $\Psi_0(\xx) \in C^\infty(\Omega) \hookrightarrow H^s(\Omega)$ for any $s > 0$. Therefore, any finite time $t \ge 0$, by~\lem{sobolev-estimates}, we have $\Psi(t) \in H^s(\Omega)$, i.e., the Schr\"odinger equation preserves the $H^s$ space. Then, by~\lem{sobolev-embedding}, we have
    \begin{align}
        \Psi(t) \in \cap_{s > d/2} H^s(\omega) \subset \cap_{k > 0} C^k(\omega) \subset C^\infty(\Omega),
    \end{align}
    which concludes the proof.
\end{proof}

\subsection{Quantum simulation of Schr\"odinger equations with non-smooth potentials}\label{append:q-sim-nonsmooth-f}

\begin{lemma}\label{lem:truncation}
    Let $\Psi(s,x)$ denote the exact solution of \eqn{schrodinger-standard-form} and $\widetilde{\Psi}(s,x)$ denote the approximated solution by the Fourier spectral method (truncated up to frequency $n$).\footnote{More details on the Fourier spectral method can be found in Section 2.2 (in particular, Lemma 1) in \cite{childs2022quantum}.}
    We assume that the initial data $\Psi_0(x)$ is periodic and analytic in $x\in \Omega$. Then, for any integer $n \ge 1$, the error from the Fourier spectral method satisfies
    \begin{align}
        \max_{x, t} |\Psi(x,t) - \widetilde{\Psi}(x,t)| \le 2r^{n/2+1},
    \end{align}
    where $0 < r < \min(\frac{1}{2},\frac{1}{Ae^{\xi}})$, $A$ is an absolute constant that only depends on $\Psi_0(x)$, and $\xi$ is the same as in~\eqn{xi-and-S}.
\end{lemma}
\begin{proof}
    We give the proof in one dimension, as the same argument is readily generalized to arbitrary finite dimensions. We assume the initial data $\Psi_0(x)$ is periodic over $[0, 2\pi]$. The analyticity implies that, for any $x \in [0, 2\pi]$, there is a constant $C$ such that
    \begin{align}
        \left|\Psi^{(k)}_0(x)\right| \le C^{k+1}(k!).
    \end{align}
    Therefore, the function $\Psi_0(x)$ admits an analytic continuation in the strip 
    $$\Gamma_0 = \left\{z\in \C: |z - x| < \frac{1}{C}, x \in [0, 2\pi]\right\}.$$
    According to the discussion in~\append{regularity}, the Sobolev norm of the wave function $\Psi(s)$ has an exponential growth:
    \begin{align}\label{eqn:sobolev}
        \left\|\Psi^{(k)}(s, \cdot)\right\| \le e^{\xi(s) S} \left\|\Psi^{(k)}_0\right\| \le e^{\xi(s) S} C^{k+1} (k!),
    \end{align}
    where $\xi(s) = \int^s_0 \varphi(r)\d r$. This implies that the wave function $\Psi(s,x)$ is periodic and analytic for any finite $s \in [0, s_f]$. The strip on which $\Psi(s,x)$ admits an analytic continuation is
    \begin{equation}
        \Gamma_s = \left\{z\in \C: |z - x| < \frac{1}{C e^{\xi(s) S}}, x \in [0, 2\pi]\right\}.
    \end{equation}
    
    Suppose that the function $\Psi(s,x)$ allows an exact, infinite trigonometric polynomial representation (see \cite[Lemma 16]{childs2022quantum}),
    \begin{align}
        \Psi(s,x) = \sum^\infty_{k=0}c_k(s) e^{ikx}.
    \end{align}
    Let $\Tilde{\Psi}(s,x)$ be the truncated Fourier series up to $k = n/2$, then the error from the Fourier spectral method satisfies
    \begin{align}
        |\Psi(s,x) - \Tilde{\Psi}(s,x)| \le \sum^\infty_{k=n/2+1}|c_k(s)|.
    \end{align}
    Since the function $\Psi(s,x)$ allows an analytic continuation in the strip $\Gamma_s$, it follows from the Riemann--Lebesgue lemma that for $k \to \infty$, we have
    \begin{equation}
        |c_k| \cdot e^{-k\frac{1}{Ce^{\xi S}}} \to 0.
    \end{equation}
    It turns out that the error from the Fourier spectral method satisfies
    \begin{align}
        |\Psi(x,t) - \Tilde{\Psi}(x,t)| \le \sum^\infty_{k=n/2+1}|c_k(t)| \le A Ce^{\xi S}e^{- \frac{n/2 + 1}{Ce^{\xi S}}}.
    \end{align}
    Note that we use the approximation $e^{-x} \approx 1-x$ for small $x$ in the last step.
\end{proof}

Now, we are ready to prove~\prop{spectral-method}.
\begin{proof}
    By~\lem{truncation}, to ensure that the simulation error is bounded by $\eta$, we may choose the truncation number 
    \begin{equation}
        n \le \mathcal{O}\left(e^{\xi S} \log(\frac{1}{\eta})\right).
    \end{equation}
    Plugging this truncation number in Equation (113) in \cite{childs2022quantum}, we prove the query complexity in \prop{spectral-method}. It is worth noting that the truncation number $n$ scales exponentially $\xi S$ (i.e., the $L^1$-norm of the time-dependent potential $\varphi(s) f$). This implies that the space complexity using quantum computers is polynomial in $\xi S$.
\end{proof}

\section{Experiment details}

\subsection{Test functions} \label{sec:A-func}
Below list 12 functions tested  in \sec{num}.
\begin{enumerate}
    \item \texttt{WF} (WF)~\cite{Polak1992}
    \begin{equation*}
        f=\max\left(\frac{1}{2}\left( x_{1} +\frac{10x_{1}}{x_{1} +0.1} +2x_{2}^{2}\right) ,\ \frac{1}{2}\left( -x_{1} +\frac{10x_{1}}{x_{1} +0.1} +2x_{2}^{2}\right) ,\ \frac{1}{2}\left( x_{1} -\frac{10x_{1}}{x_{1} +0.1} -2x_{2}^{2}\right)\right)
    \end{equation*}
    tested on $[ -10,\ 10]^{2}$ with globally optimal solution at $f( 0,\ 0) =0$.
    \item \texttt{CROWNEDCROSS} (Crowned Cross)~\cite{gavana}
    \begin{equation*}
        f=0.0001\left( |\sin( x_{1})\sin( x_{2})\exp\left( 100-\frac{\sqrt{x_{1}^{2} +x_{2}^{2}}}{\pi }\right) |+1\right)^{0.1}
    \end{equation*}
    tested on $[ -10,\ 15]^{2}$ with globally optimal solution at $f( 0,\ 0) =0.0001$.
    \item \texttt{BUKIN06} (Bukin 6)~\cite{gavana}
    \begin{equation*}
        f=100\sqrt{|x_{2} -0.01x_{1}^{2} |} +0.01|x_{1} +10|
    \end{equation*}
    tested on $[ -15,\ -5] \times [ -3,\ 3]$ with globally optimal solution at $f( -10,\ 1) =0$.
    \item \texttt{KEANE} (Keane)~\cite{baronti2024python}
    \begin{equation*}
         f=-\frac{|\cos^{4}( x_{1}) +\cos^{4}( x_{2}) -2\cos^{2}( x_{1})\cos^{2}( x_{2}) |}{\sqrt{x_{1}^{2} +2x_{2}^{2}}}
    \end{equation*}
    tested on $[10^{-8},\ 10]^2$ with the globally optimal solution at $f(1.60086, 0.468498) = 0.673207$.
    \item \texttt{SCHWEFEL} (Schwefel)~\cite{gavana}
    \begin{equation*}
         f=418.9828872724336-x_{1}\sin\left(\sqrt{|x_{1} |}\right)
    \end{equation*}
    tested on $[ -500,\ 500]$ with globally optimal solution at $f( 420.9687474737558) =0$.
    \item \texttt{ACKLEY} (Ackley)~\cite{gavana}
    \begin{equation*}
         f=-20\exp\left( -0.2\sqrt{\frac{x_{1}^{2} +x_{2}^{2}}{2}}\right) -\exp\left(\frac{\cos( 2\pi x_{1}) +\cos( 2\pi x_{2})}{2}\right) +20+e
    \end{equation*}
    tested on $[ -15,\ 30]^{2}$ with globally optimal solution at $f( 0,\ 0) =0$.
    \item \texttt{XINSHEYANG04} (Xinshe Yang 4)~\cite{gavana}
    \begin{equation*}
         f=\left(\sin^{2}( x_{1}) +\sin^{2}( x_{2}) -\exp\left( -x_{1}^{2} -x_{2}^{2}\right)\right)\exp\left( -\sin^{2}\left(\sqrt{|x_{1} |}\right) -\sin^{2}\left(\sqrt{|x_{2} |}\right)\right)
    \end{equation*}
    tested on $[ -10,\ 10]^{2}$ with globally optimal solution at $f( 0,\ 0) =-1$.
    \item \texttt{CARROMTABLE} (Carrom Table)~\cite{nalaie2015}
    \begin{equation*}
         f=-\frac{1}{30}\exp\left( \left|2-\frac{2}{\pi }\sqrt{x_{1}^{2} +x{_{2}^{2}}} \right|\right)\cos^{2}( x_{1})\cos^{2}( x_{2})
    \end{equation*}
    tested on $[ -10,\ 10]^{2}$ with globally optimal solution at $f=( \pm 9.646157266349,\ \pm 9.646157266349) =-24.1568155165$.
    \item \texttt{RANA} (Rana)~\cite{gavana}
    \begin{align*}
        f = & \, x_{1}\sin\left(\sqrt{|x_{2} -x_{1} +1|}\right)\cos\left(\sqrt{|x_{2} +x_{1} +1|}\right) \\
        & +( x_{2} +1)\sin\left(\sqrt{|x_{2} +x_{1} +1|}\right)\cos\left(\sqrt{|x_{2} -x_{1} +1|}\right)
    \end{align*}
    tested on $[ -500,\ 500]^{2}$ with globally optimal solution at $f( -300.3376328023,\ 500) =-500.802160296664$.
    \item \texttt{DROPWAVE} (Drop Wave)~\cite{nalaie2015}
    \begin{align*}
        f=-\frac{1+\cos\left( 12\sqrt{x_{1}^{2} +x_{2}^{2} +x_{3}^{2}}\right)}{2+0.5\left( x_{1}^{2} +x_{2}^{2} +x_{3}^{2}\right)}
    \end{align*}
    tested on $[ -5.12,\ 5.12]^{3}$ with globally optimal solution at $f( 0,\ 0,\ 0) =-1$.
    \item \texttt{LAYEB04} (Layeb 4)~\cite{layeb2022}
    \begin{align*}
        f=\sum _{i=1}^{2}(\ln( |x_{i} x_{i+1} |+0.001) +\cos( x_{1} +x_{i+1}))
    \end{align*}
    tested on $[ -10,\ 10]^{3}$ with globally optimal solution at $f( 0,\ ( 2j-1) \pi ,\ 0) =2\ln( 0.001) -2$, where $j\in \{-1,\ 0,\ 1,\ 2\}$.
    \item \texttt{DAMAVANDI} (Damavandi)~\cite{gavana}
    \begin{align*}
        f=\left( 1-\left|\frac{\sin( \pi ( x_{1} -2))\sin( \pi ( x_{2} -2))}{\pi ^{2}( x_{1} -2)( x_{2} -2)} \right|^{5}\right)\left( 2+( x_{1} -7)^{2} +2( x_{2} -7)^{2}\right)
    \end{align*}
    tested on $[0, 14]^2$ with the globally optimal solution at $f(2 + 10^{-10}, 2 + 10^{-10}) = 0$.
\end{enumerate}

In the numerical experiment (\sec{num}), we further augment these functions with a continuous barrier that steeply diverges \emph{out} of the original domain. 
This adjustment is made because the discretized QHD algorithm described in \sec{disc-qhd} implicitly alters the topology of the domain $[-L,L]^d$ due to its handling of boundary conditions. 
To be more specific, \algo{discrete-time-qhd} works with a \emph{periodic} boundary condition, effectively allowing the wave function to propagate across boundaries of the domain as if the topology is $(\mathbb{R} / 2L\mathbb{Z})^d$.
In contrast, classical optimization methods such as Subgrad and LFMSGD do not have this feature.
Arguably, this adjustment will not affect the performance of Subgrad and LFMSGD by their design.

\subsection{Effect of changing \texorpdfstring{$L$}{L} in QHD} \label{sec:A-qhd-L}
We analyze the effect of the scaling parameter $L$, introduced in \sec{meth}, on the QHD dynamics.  
Without loss of generality, assume that the domain of a test function \( f(\xx) \) is \( [-1,1]^d \).  
By rescaling the domain via \( \yy = L\xx \), we transform \( f(\xx) \) into \( g(\yy) = f(\yy/L) \) with domain \( [-L,L]^d \), which is then used in QHD.  

The QHD evolution equation in \eqn{qhd-nc} is  
\begin{equation} \label{eqn:qhd-nc-L1}
    i \partial_t \Psi(t) = H(t)\Psi(t),\quad
        H(t) = \frac{1}{\lambda(t)} \left(-\frac{1}{2}\Delta_{\yy} \right) + \lambda(t) g(\yy),
\end{equation}
where the Laplacian operator is given by \( \Delta_{\yy} \coloneqq \sum_{i=1}^{d} \frac{\d^2}{\d y_i^2} \).  
Using the chain rule, the Laplacian transforms as  
\begin{equation}
\Delta_{\yy} = L^{-2} \Delta_{\xx}.
\end{equation}
Substituting this into \eqn{qhd-nc-L1}, we obtain  
\begin{equation} \label{eqn:qhd-nc-L2}
    i L \partial_t \Psi(t) = \left( \frac{1}{L\lambda(t)} \left(-\frac{1}{2}\Delta_{\xx} \right) + L\lambda(t) f(\xx)\right)\Psi(t).
\end{equation}  

Now, introducing the rescaled time variable \( \tau \coloneqq t/L \) and defining \( \Phi(\tau) = \Psi(t) = \Psi(L\tau) \), we rewrite \eqn{qhd-nc-L2} as  
\begin{equation} \label{eqn:qhd-nc-L3}
    i \partial_\tau \Phi(\tau) = \left( \frac{1}{L\lambda(L \tau)} \left(-\frac{1}{2}\Delta_{\xx} \right) + L\lambda(L \tau) f(\xx)\right)\Phi(\tau).
\end{equation}  
Therefore, the equation \eqn{qhd-nc-L3} represents the same QHD dynamics as in \eqn{qhd-nc}, but with the rescaled time variable $\tau=t/L$ and the coefficient function $\mu(\tau) \coloneqq L \lambda(L \tau) $:
\begin{equation} \label{eqn:qhd-nc-L4}
    i \partial_\tau \Phi(\tau) = \left( \frac{1}{\mu(\tau)} \left(-\frac{1}{2}\Delta_{\xx} \right) + \mu(\tau) f(\xx)\right)\Phi(\tau).
\end{equation}

\subsection{Learning-rate free momentum stochastic gradient descent (LFMSGD)} \label{sec:A-classical-algo}
LFMSGD~\cite{hu2024learning} is a subgradient method that generalizes previous learning-rate free subgradient methods like DoG (Distance over Gradients)~\cite{ivgi2023dog} and DoWG (Distance over Weighted Gradients)~\cite{khaled2023dowg}.
It starts with an initial point $\xx_0$ and the update rule is
\begin{subequations}
\begin{align}
    \eta_t &= \frac{\mu_t}{\sqrt{\epsilon_0 + \sum_{i=1}^{t} \| m_i \|^2}}, \\
    m_{t+1} &= \beta m_t + (1 - \beta) g(\xx_t), \\
    \xx_{t+1} &= \xx_t - \eta_t m_{t+1},
\end{align}
\end{subequations}
where $g(\xx_t)$ is an \emph{estimate} of $\partial f(\xx_t)$, $\mu_t = \max_{i \leq t} \{ \norm{\xx_i - \xx_0} \}$, $0 < \epsilon_0 \ll 1$ is a small prefixed constant that enforces the numerical stabilization and $\beta \in [0,1)$ is the momentum parameter.
The returned solution will be $\xx^\ast \coloneqq \xx_T$ where $T$ is the total number of steps.

We choose the parameter $\beta = 0.9$ as recommended in~\cite{hu2024learning}.
The stochasticity in the original LFMSGD arises from the common setting in machine learning where $f = \frac{1}{N} \sum_{1 \leq i \leq N} f_i$ and a natural, computationally efficient approach is to choose $g$ by cycling through the functions $\partial f_1,\partial f_2,\dots,\partial f_N$ in sequence for each $t$.
In our case we do not have a natrual approximation of $\partial f$, so we add Gaussian noise to model $g$, following~\cite{leng2023qhd,shi2020learning}:
\begin{equation}
    g(\xx) = \partial f(\xx) + \mathcal{N}(\boldsymbol{0},\sigma^2).
\end{equation}

\begin{figure}[htbp!]
        \centering
        \includegraphics[height=0.9\textheight]{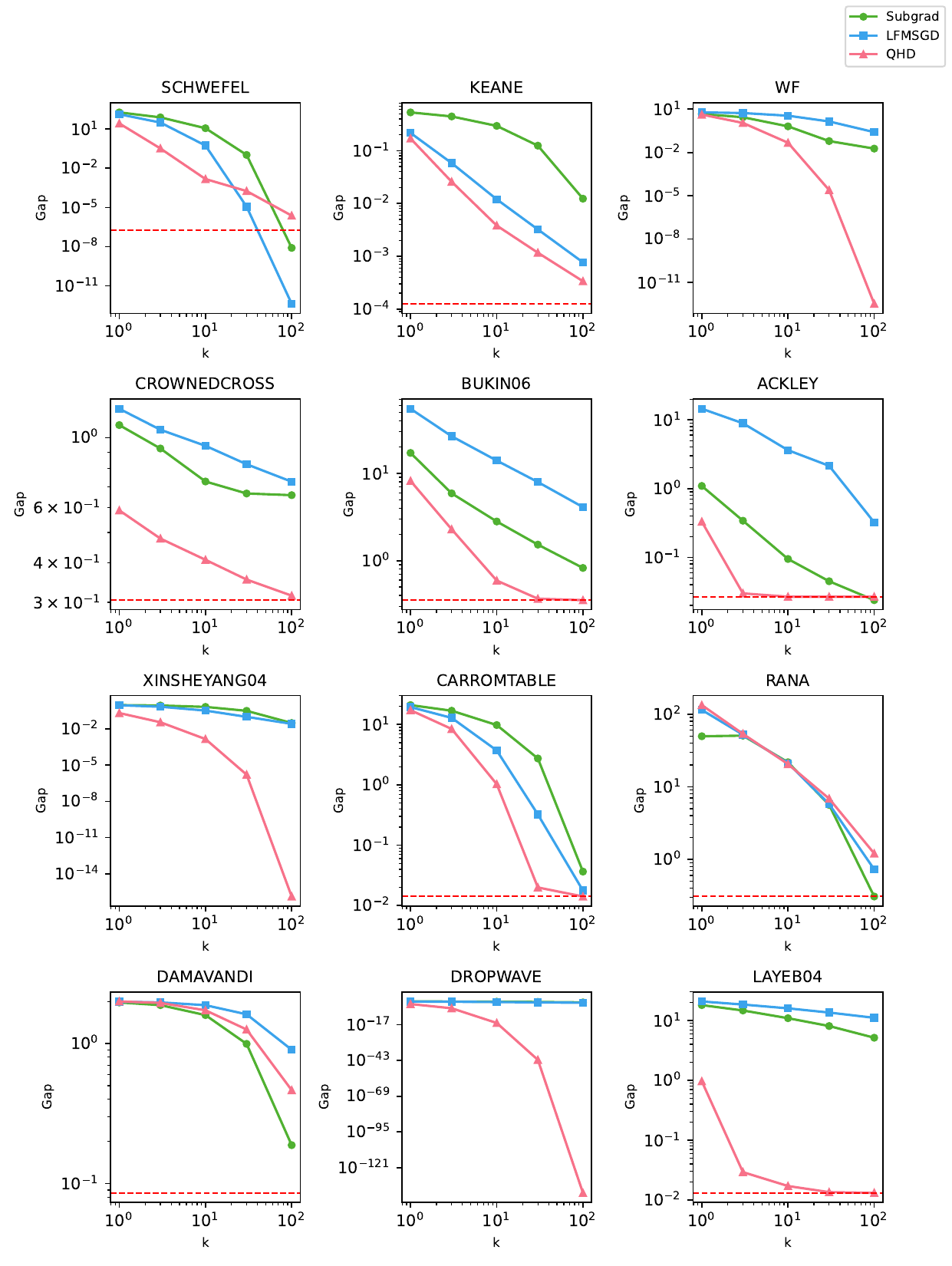} 
        \caption{Best-of-$k$ optimality gap for 12 test functions. Minimal optimality gaps that QHD can ``see'' after discretization are shown as red dotted lines.}
        \label{fig:best-of-k-with-min}
\end{figure}

\subsection{Resolution bottleneck of QHD simulation} \label{sec:A-qhd-bn}

Recall that we choose spatial discretization number $N=512$ for QHD in \sec{num}.
It turns out that for some test functions, the resolution is not enough for QHD to even see a point that has smaller enough optimality gap.
In \fig{best-of-k-with-min} (cf. \fig{best-of-k}), we highlight the minimal optimality gap that QHD can ``see'' after discretization with red dotted lines.
It can be inferred that for functions like \texttt{ACKLEY}, QHD performance no longer improves with $k$ increasing because of this resolution bottleneck.

\end{document}